\documentclass[12pt,a4paper]{article}

\usepackage{latexsym}
\usepackage{amsmath}
\usepackage{amssymb}
\usepackage{amsthm}
\usepackage{amscd}
\usepackage{mathrsfs}
\usepackage[all]{xy}
\usepackage{graphicx}
\usepackage{bm}
\usepackage{comment}

\def\cprime{$'$}

\newtheorem{thm}{Theorem}[section]
\newtheorem{prop}[thm]{Proposition}
\newtheorem{defn}[thm]{Definition}
\newtheorem{lem}[thm]{Lemma}
\newtheorem{cor}[thm]{Corollary}

\newtheorem{rem}[thm]{Remark}

\newtheorem*{thmn}{Theorem}

\theoremstyle{remark}

\makeatletter

\@addtoreset{equation}{section}
\makeatother

\newcommand{\cf}{\textit{cf.\ }}

\DeclareMathOperator{\Spec}{Spec}

\DeclareMathOperator{\Gal}{Gal}
\DeclareMathOperator{\id}{id}

\DeclareMathOperator{\End}{End}
\DeclareMathOperator{\Hom}{Hom}

\DeclareMathOperator{\Spf}{Spf}

\DeclareMathOperator{\Ind}{Ind} 

\DeclareMathOperator{\Nr}{Nr}
\DeclareMathOperator{\Trd}{Trd}
\DeclareMathOperator{\Nrd}{Nrd}

\DeclareMathOperator{\Spa}{Spa}
\DeclareMathOperator{\Lie}{Lie}
\DeclareMathOperator{\Art}{Art}
\DeclareMathOperator{\sgn}{sgn}

\newcommand{\bom}[1]{\mbox{\boldmath $#1$}}

\renewcommand{\eqref}[1]{(\ref{#1})}

\renewcommand{\bigskip}{\vspace{0.2cm}}

\newcommand{\bA}{\mathbb{A}}

\newcommand{\bF}{\mathbb{F}}

\newcommand{\bQ}{\mathbb{Q}}

\newcommand{\bZ}{\mathbb{Z}}

\newcommand{\bfC}{\mathbf{C}}

\newcommand{\bfM}{\mathbf{M}}

\newcommand{\bfg}{\mathbf{g}}

\newcommand{\cB}{\mathcal{B}}
\newcommand{\cC}{\mathcal{C}}
\newcommand{\cD}{\mathcal{D}}

\newcommand{\cG}{\mathcal{G}}

\newcommand{\cM}{\mathcal{M}}

\newcommand{\cO}{\mathcal{O}}

\newcommand{\cS}{\mathcal{S}}

\newcommand{\cX}{\mathcal{X}}

\newcommand{\fp}{\mathfrak{p}}

\newcommand{\fI}{\mathfrak{I}}

\newcommand{\fX}{\mathfrak{X}}

\newcommand{\iGL}{\mathit{GL}}

\newcommand{\rmur}{\mathrm{ur}}
\newcommand{\rmac}{\mathrm{ac}}

\newcommand{\ol}{\overline}
\newcommand{\wh}{\widehat}

\begin{document}
\title{Affinoids in the Lubin-Tate perfectoid space and\\ simple supercuspidal representations I: tame case}
\author{Naoki Imai and Takahiro Tsushima}
\date{}
\maketitle


\begin{abstract} 
We construct a family of affinoids in the Lubin-Tate perfectoid space 
and their formal models 
such that 
the middle cohomology of the reductions of the formal models 
realizes the 
local Langlands correspondence and the 
local Jacquet-Langlands correspondence 
for simple supercuspidal representations 
in the case where 
the dimension of Galois representations 
is prime to the residue characteristic. 
The reductions of the formal models are 
isomorphic to the perfections of Artin-Schreier varieties 
associated to quadratic forms. 
\end{abstract}

\section*{Introduction}
Let $K$ be a non-archimedean local field with 
residue field $k$. 
Let $p$ be the characteristic of $k$. 
We write $\mathcal{O}_K$ for the ring of integers of $K$. 
We fix an algebraic closure $k^{\mathrm{ac}}$ of $k$. 
Let $n$ be a positive integer. 
The Lubin-Tate spaces are deformation spaces of 
the one-dimensional formal $\mathcal{O}_K$-module 
over $k^{\mathrm{ac}}$ of height $n$ with level structures. 
We take a prime number $\ell$ that is different from $p$. 
The local Langlands correspondence (LLC) and the 
local Jacquet-Langlands correspondence (LJLC) 
for supercuspidal representations of $\mathit{GL}_n (K)$ 
are realized in the $\ell$-adic cohomology of the Lubin-Tate spaces. 
This was proved by Boyer 
in \cite{BoyMDr} in the equal characteristic case, and 
by Harris-Taylor in \cite{HTsimSh} 
in the mixed characteristic case. 
However, the proofs rely on global automorphic arguments, 
and the geometry of Lubin-Tate spaces is 
still mysterious. 

As a geometric study of Lubin-Tate spaces, 
Yoshida constructed a semi-stable model of 
the Lubin-Tate space with a full level $\mathfrak{p}$-structure in 
\cite{YoLTv}, 
where $\mathfrak{p}$ is the maximal ideal of $\mathcal{O}_K$. 
Further, he showed that 
Deligne-Lusztig varieties for $\mathit{GL}_n (k)$ 
appear as a Zariski open subset of the reduction of the semi-stable model, 
and that the cohomology of the reduction 
realizes the LLC for depth zero supercuspidal representations. 
We say that a supercuspidal representation of $\mathit{GL}_n (K)$ is 
of unramified type 
if its Weil parameter is induced from 
a character of the Weil group of the unramified extension of $K$ 
of degree $n$, 
and of ramified type if it is not of unramified type. 
After the work of Yoshida, 
Weinstein constructed in \cite{WeGood} 
a family of affinoids in a Lubin-Tate space with 
a finite level in the equal characteristic case 
such that 
the cohomology of the reductions of the affinoids realizes 
the LLC for depth one supercuspidal representations 
of unramified type. 
This work was generalized to higher 
depth supercuspidal representations of unramified type 
in any characteristic by 
Boyarchenko-Weinstein in \cite{BWMax} 
in the Lubin-Tate perfectoid space setting, 
where the Lubin-Tate perfectoid space is a 
Lubin-Tate space with an infinite level in some sense. 
In these cases, 
analogues of Deligne-Lusztig varieties appear as the reductions 
of formal models of the affinoids. 
On the other hand, 
even a conjecture was not known on 
what kind of varieties realize 
the LLC for supercuspidal representations of ramified type 
if $n >2$. 

In a series of papers, 
we construct a family of affinoids 
in the Lubin-Tate perfectoid space 
and their formal models such that 
the cohomology of the reductions of the models 
realizes the LLC and the LJLC for 
the representations whose exponential Swan conductors equal one, 
which we call simple supercuspidal representations in this paper 
(\cf \cite{ALssrGL}). 
This generalizes results in \cite{ITstab3} and \cite{ITreal3} to 
higher dimensional cases in the perfectoid space setting. 
Note that the supercuspidal representations of $\iGL_n (K)$ 
of exponential Swan conductors one 
are called epipelagic in \cite{BHLepi} after 
a terminology in \cite{RYEinv}. 
The simple supercuspidal representations are the first parts of 
representations of ramified type in some sense. 
We say that 
a representation is essentially simple supercuspidal 
if it is a character twist of a simple supercuspidal representation. 

In this paper, we treat a tame case, which means the 
case where $n$ is prime to $p$. 
Let 
$q$ be the number of the elements of $k$ and 
$D$ be the central division algebra over $K$ of invariant $1/n$. 
Let $W_K$ denote the Weil group of $K$. 
The main theorem is the following: 
\begin{thmn}
For a totally tamely ramified extension $L$ of 
$K$ of degree $n$, 
there is an affinoid $\cX^L$ in the Lubin-Tate perfectoid space 
and its formal model $\fX^L$ such that 
\begin{itemize}
\item 
the special fiber $\ol{\fX^L}$ of $\fX^L$ is isomorphic to 
the perfection of the affine smooth variety defined by 
$z^q -z =\sum_{1 \leq i \leq j \leq n-1} y_i y_j$ 
in $\bA_{k^{\rmac}}^n$,  
\item 
the stabilizer $H_L \subset \iGL_n (K) \times D^{\times} \times W_K$ 
of $\cX^L$ naturally acts on $\ol{\fX^L}$, and 
\item 
$\mathrm{c\mathchar`-Ind}_{H_L}^{\iGL_n (K) \times D^{\times} \times W_K} 
 H_{\mathrm{c}}^{n-1}(\overline{\fX^L},\ol{\bQ}_{\ell} )$ 
 realizes the LLC and the LJLC for 
 essentially simple supercuspidal representations. 
\end{itemize}
\end{thmn}
See Theorem \ref{redmod} and Theorem \ref{Pireal} for precise statements. 
In the tame case, 
an essentially simple supercuspidal representation is essentially tame, 
and its Weil parameter 
is induced from a character of the Weil group of 
a totally tamely ramified extension of $K$ of degree $n$, 
which appears as $L$ in the above theorem. 
Explicit descriptions of the LLC and the LJLC for 
essentially tame representations are given by 
Bushnell-Henniart in 
\cite{BHestI}, \cite{BHestII} and \cite{BHestIII}, and 
in \cite{BHestJL} respectively. 
Although the characteristic of a non-archimedean local field 
is assumed to be zero in their results, 
the assumption is removed by the work \cite{HLCinon} of Henniart-Lemaire. 
See also ``Note on characteristic'' in \cite[p.~8]{BHeffL}. 
We can use their results to show 
the realization of the LLC and the LJLC 
in the cohomology of the reductions. 

In Section \ref{LTps}, 
we recall on the Lubin-Tate perfectoid spaces. 
In Subsection \ref{LTpsfm}, 
we recall 
a definition of the Lubin-Tate perfectoid space and 
a structure theorem on 
a formal model of the Lubin-Tate perfectoid space from \cite{WeSemi}. 
We also give an approximation lemma for 
a defining equation of the formal model. 
In Subsection \ref{gpfm}, we recall a group action on 
the formal model. 
In Subsection \ref{CMpts}, we recall on 
CM points. 

In Section \ref{GoodRedAff}, 
we construct and study an affinoid associated to 
a totally tamely ramified extension $L$ of $K$ of degree $n$. 
In Subsection \ref{ConstAff}, 
we construct a CM point that has multiplication by $L$. 
Using the CM point, we give a construction of an affinoid. 
In Subsection \ref{RedAff}, 
we construct a formal model of the affinoid and 
study its reduction. 
As a result, we find that 
the reduction is isomorphic to 
an Artin-Schreier variety 
associated to a quadratic form. 
In Section \ref{GpX}, 
we study the group action on the reductions, 
and determine the stabilizer of the affinoid. 

In Section \ref{CohASv}, 
we study the cohomology of 
an Artin-Schreier variety $X$ associated to a quadratic form 
in a bit more general situation 
with consideration to applications to the wild case. 
If $p$ is odd, 
we can calculate the cohomology of $X$ 
by diagonalizing the quadratic form. 
If $p=2$, we need other methods. 
In fact, 
we construct a variety $X'$ that is 
purely inseparable to $X$, and 
a nice fibration 
of $X'$ over an affine space. 
Using this fibration, 
we can calculate the cohomology of $X$. 
Unfortunately, this fibration is not 
preserved by a group action. 
To calculate the group action on the cohomology of $X$, 
we use an argument changing $\ell$ in the coefficient, 
which relies on an $\ell$-independence result on 
the trace of an action on an $\ell$-adic cohomology. 

In Section \ref{RealLLCLJLC}, 
we show the realization of 
the LLC and the LJLC in the cohomology of the reductions. 
In Subsection \ref{ExpCorr}, 
we recall 
explicit descriptions of the LLC and 
the LJLC given by Bushnell-Henniart. 
In Subsection \ref{subsecReal}, 
we prove Theorem \ref{Pireal} 
assuming Proposition \ref{ky}, 
which is a formula comparing 
the Langlands constant of 
a totally ramified extension of degree $n$ 
with a quadratic Gauss sum. 
In Subsection \ref{PrProp}, 
we prove Proposition \ref{ky} by induction on the $2$-adic valuation of $n$, 
using the quadratic reciprocity law. 

In a subsequent paper \cite{ITsimpwild}, 
we will study the geometric realization of 
the LLC and the LJLC for simple supercuspidal representations in 
the wild case. 

After this work was completed, 
Tokimoto generalizes the construction of affinoids 
to cases for some essentially tame representations of higher depth 
in the positive characteristic case 
using results in this paper (\cf \cite{TokLTs}). 
In Remark \ref{rem:compari}, 
we compare the construction in this paper with 
that in \cite{TokLTs} to see 
that our construction using CM points 
naturally fits into a systematic 
description of general phenomena in his paper.

\subsection*{Acknowledgements}
The authors are grateful to Yoichi Mieda and Kazuki Tokimoto 
for a lot of helpful comments on a previous version of this paper. 
The authors would like to thank referees for suggestions for improvements. 
This work was supported by JSPS KAKENHI Grant Numbers 
26707003, 15K17506. 

\subsection*{Notation}
For a non-archimedean valuation field $F$, 
its valuation ring is denoted by $\mathcal{O}_F$. 
For $a \in \mathbb{Q}$ and 
elements $f$, $g$ with valuation $v$ that takes 
values in $\mathbb{Q}$, 
we write 
$f \equiv g \mod a$ 
if $v(f-g) \geq a$, and  
$f \equiv g  \mod\!_{>}\, a$ 
if $v(f-g) >a$. 
For a topological field extension $E$ over $F$, 
let $\Gal (E/F)$ denote the group of 
the continuous automorphisms of $E$ over $F$. 

\section{Lubin-Tate perfectoid space}\label{LTps}
\subsection{Lubin-Tate perfectoid space and its formal model}\label{LTpsfm}
Let $K$ be a non-archimedean local field with 
a residue field $k$ of characteristic $p$. 
Let $q$ the number of the elements of $k$. 
We write $\mathfrak{p}$ for the maximal ideal of $\mathcal{O}_K$. 
We fix an algebraic closure 
$K^{\mathrm{ac}}$ of $K$. 
Let 
$k^{\mathrm{ac}}$ be the residue field of 
$K^{\mathrm{ac}}$. 

Let $n$ be a positive integer. 
We take a one-dimensional formal 
$\mathcal{O}_K$-module $\cG_0$ over $k^{\mathrm{ac}}$ 
of height $n$, 
which is unique up to isomorphism. 
Let $K^{\mathrm{ur}}$
be the maximal unramified extension of $K$ in $K^{\mathrm{ac}}$.
We write $\widehat{K}^{\mathrm{ur}}$ for the completion of $K^{\mathrm{ur}}$. 
Let $\mathcal{C}$ be the category of 
complete Noetherian local 
$\mathcal{O}_{\widehat{K}^{\mathrm{ur}}}$-algebras 
with residue field $k^{\mathrm{ac}}$.

\begin{defn}
A deformation of 
$\cG_0$ to $R \in \mathcal{C}$ 
means a pair $(\cG,\iota)$, where 
$\cG$ is a 
formal $\mathcal{O}_K$-module 
over $R$ and 
$\iota \colon \cG_0 \to \cG\otimes_R k^{\mathrm{ac}}$ is 
an isomorphism. 
\end{defn}

Let $\cG$ be a formal $\cO_K$-module over $R \in \cC$. 
For $a \in \mathcal{O}_K$, 
let $[a]_{\cG} \colon \cG \to \cG$ be the 
multiplication by $a$, and 
$\cG[a]$ be the kernel of $[a]_{\cG}$. 
For an integer $m \geq 0$, 
we define $\cG[\mathfrak{p}^m]$ to be $\cG[a]$ for some 
$a \in \mathfrak{p}^{m} \setminus \mathfrak{p}^{m+1}$. 

\begin{defn}
A Drinfeld level $\mathfrak{p}^m$-structure on $\cG$ means 
an $\mathcal{O}_K$-module homomorphism 
$\phi \colon (\fp^{-m}/\cO_K)^{n} \to \cG[\fp^m](R)$ 
that gives a full set of sections of $\cG[\fp^m](R)$ 
in the sense of \cite[(1.8.2)]{KMmod}. 
\end{defn}

We consider 
the functor $\mathcal{C} \to \mathbf{Sets}$ 
which associates to an object $R \in \mathcal{C}$ 
the set of isomorphism classes of triples 
$(\cG, \phi,\iota)$, where $(\cG, \iota)$
is a deformation of $\cG_0$ to $R$ and 
$\phi$ is a Drinfeld level $\mathfrak{p}^m$-structure on $\cG$. 
This functor is represented 
by a regular local ring $A_m$ by \cite[Proposition 4.3]{DrEmod}. 
Then, $\{A_m\}_{m \geq 0}$ makes an inductive 
system. 
Let $I$ the ideal of 
$\varinjlim A_m$ generated by 
the maximal ideal of $A_0$. 
Let $A$ be the $I$-adic 
completion of $\varinjlim A_m$. 
We regard $A$
as a topological ring 
by considering the 
$I$-adic topology on $A$. 
We set $\bfM_{\cG_0, \infty}=\Spf A$. 

Let $K^{\mathrm{ab}}$ be the maximal abelian 
extension of $K$ in $K^{\mathrm{ac}}$. 
We write 
$\widehat{K}^{\mathrm{ab}}$ for the completion of 
$K^{\mathrm{ab}}$. 
Let $\wedge \cG_0$ 
denote the one-dimensional formal 
$\mathcal{O}_K$-module over $k^{\mathrm{ac}}$ of height one. 
Then we have 
$\bfM_{ \wedge \cG_0, \infty} \simeq 
 \Spf \mathcal{O}_{\widehat{K}^{\mathrm{ab}}}$ 
by the Lubin-Tate theory. 
We have a determinant morphism 
\begin{equation}\label{mor}
 \bfM_{\cG_0, \infty} \to 
 \bfM_{\wedge \cG_0, \infty}, 
\end{equation} 
which is given in \cite[2.5 and 2.7]{WeSemi} 
based on \cite{HedPhD}. 
Then, by \eqref{mor}, we have the ring 
homomorphism 
$\mathcal{O}_{\widehat{K}^{\mathrm{ab}}} \to A$. 

Let $\Spa (A,A)$ 
denote the set of equivalent 
classes of 
continuous valuations $x$ 
on $A$ such that $|f(x)| \leq 1$ for any $f \in A$ 
(cf.\ \cite[3]{HuCv}). 
We fix  a uniformizer 
$\varpi$ of $\mathcal{O}_K$. 
We put 
\[
 \mathcal{M}_{\infty}=\bigl\{
 x \in \Spa  (A,A) \bigm| |\varpi(x)| \neq 0 
 \bigr\}. 
\]
Then $\cM_{\infty}$ 
naturally has a structure of 
an adic space over $\wh{K}^{\rmur}$. 
For a deformation $\cG$ of $\cG_0$ over $\cO_{\bfC}$, 
we put 
\[
 V_{\fp} (\cG) =\bigl( \varprojlim \cG[\fp^m](\cO_{\bfC}) \bigr) 
 \otimes_{\cO_K} K, 
\]
where the transition maps are multiplications by $\varpi$. 
By the construction, each point of 
$\cM_{\infty}(\bfC)$ 
corresponds to a triple $(\cG,\phi,\iota)$ 
that consists of 
a formal $\cO_K$-module over $\cO_{\bfC}$, 
an isomorphism $\phi \colon K^n \to V_{\fp} (\cG)$ and 
an isomorphism 
$\iota \colon \cG_0 \to \cG \otimes_{\cO_{\bfC}} k^{\rmac}$ 
(cf.~\cite[Definition 2.10.1]{BWMax}). 

By the ring homomorphism 
$\mathcal{O}_{\widehat{K}^{\mathrm{ab}}} \to A$, 
we can regard $\cM_{\infty}$ as an 
adic space over 
$\eta=\Spa (\widehat{K}^{\mathrm{ab}}, 
 \mathcal{O}_{\widehat{K}^{\mathrm{ab}}})$, 
for which we write 
$\mathcal{M}_{\infty, \eta}$. 
Let $\mathbf{C}$ be the completion of $K^{\mathrm{ac}}$, 
and $\bar{\eta}=\Spa (\mathbf{C}, \mathcal{O}_{\mathbf{C}})$. 
We have a natural embedding 
$\widehat{K}^{\mathrm{ab}} \hookrightarrow \mathbf{C}$. 
We put 
\[
 \mathcal{M}_{\infty, \overline{\eta}} = 
 \mathcal{M}_{\infty, \eta} \times_{\eta} \ol{\eta}. 
\] 
Then, 
$\mathcal{M}_{\infty, \overline{\eta}}$ 
is a perfectoid space over $\mathbf{C}$ 
in the sense of 
\cite[Definition 6.15]{SchPerf} 
by \cite[Lemma 2.10.1]{WeSemi}. 
We call 
$\mathcal{M}_{\infty, \overline{\eta}}$ 
the Lubin-Tate perfectoid space. 

In the following, 
we recall an explicit description of 
$A^{\circ}=A \widehat{\otimes}_{\mathcal{O}_{\widehat{K}^{\mathrm{ab}}}} 
 \mathcal{O}_{\mathbf{C}}$ 
given in \cite[(2.9.2)]{WeSemi}. 
Let $\wh{\cG}_0$ be the formal $\cO_K$-module 
over $\cO_K$ whose logarithm is 
$\sum_{i=0}^{\infty} \frac{X^{q^{in}}}{\varpi^i}$ (cf.~\cite[2.3]{BWMax}). 
Let $\cG_0$ be the formal $\cO_K$-module over 
$k^{\rmac}$ obtained as reduction of $\wh{\cG}_0$. 
We put $\mathcal{O}_D =\End \cG_0$ and 
$D=\mathcal{O}_D \otimes_{\cO_K} K$, which is 
the central division algebra over $K$ of invariant $1/n$. 
Let $[\ \cdot \ ]$ denote the action of $\mathcal{O}_D$ on $\cG_0$. 
Let $\varphi$ be the element of $D$ such that $[\varphi](X)=X^q$. 
Let $K_n$ 
be the unramified extension of $K$ of degree $n$. 
For an element $a \in \mathcal{O}_{\mathbf{C}}$, 
its image in the residue field is denoted by $\bar{a}$. 
We consider the $K$-algebra embedding of $K_n$ 
into $D$ determined by 
\[
 [\zeta](X)=\bar{\zeta} X \ \ 
 \text{for} \ \ \zeta \in \mu_{q^n -1}(K_n). 
\]
Then we have $\varphi^n=\varpi$ and 
$\varphi \zeta =\zeta^q \varphi$ for $\zeta \in \mu_{q^n -1} (K_n)$. 
Let $\wh{\wedge \cG_0}$ 
be the one-dimensional formal $\mathcal{O}_K$-module 
over $\cO_K$ whose logarithm is 
$\sum_{i=0}^{\infty} (-1)^{(n-1)i} \frac{X^{q^i}}{\varpi^i}$. 
We choose a compatible system $\{t_m\}_{m \geq 1}$
such that 
\begin{equation}\label{tmcho}
 t_m \in {K}^{\mathrm{ac}} \quad (m \geq 1), \quad 
 t_1 \neq 0, \quad 
 [\varpi]_{\wh{\wedge \cG_0}}(t_1)=0, \quad 
 [\varpi]_{\wh{\wedge \cG_0}}(t_m)=t_{m-1} \quad (m \geq 2).
\end{equation}
We put
\[
 t=\lim_{m \to \infty}(-1)^{q(n-1)(m-1)} t_m^{q^{m-1}} 
 \in \mathcal{O}_{\mathbf{C}}.
\]
Let $v$ be the normalized valuation of $K$ 
such that $v(\varpi)=1$. 
The valuation $v$ naturally extends to a valuation on 
$\mathbf{C}$, for which we again write $v$. 
Note that $v(t)=1/(q-1)$. 
For an integer $i \geq 0$, we put 
$t^{q^{-i}} =\lim_{m \to \infty} (-1)^{q(n-1)(m-1)} t_m^{q^{m-i-1}}$. 


Let $W_K$ be the 
Weil group of $K$. 
Let 
$\mathrm{Art}_K 
\colon K^{\times} \xrightarrow{\sim} W_K^{\mathrm{ab}}$
be the Artin 
reciprocity map normalized 
such that a uniformizer is sent to a
lift of the geometric Frobenius element. 
We use similar normalizations also for the Artin reciprocity maps 
for other non-archimedean local fields. 
Let $\sigma \in W_{K}$. 
Let $n_{\sigma}$ 
be the image of $\sigma$ under the composite 
\[
 W_K \twoheadrightarrow W_K^{\mathrm{ab}} 
 \xrightarrow{\mathrm{Art}^{-1}_K}
 K^{\times} \xrightarrow{v} \mathbb{Z}. 
\]
Let $a_K \colon W_K \to \mathcal{O}_K^{\times}$ 
be the homomorphism given by the action of $W_K$ 
on $\{t_m\}_{m \geq 1}$. 
It induces an isomorphism 
$a_K \colon 
\mathrm{Gal}(\widehat{K}^{\mathrm{ab}}/\widehat{K}^{\mathrm{ur}}) \simeq 
\mathcal{O}_K^{\times}$. 

For $m \geq 0$, we put 
\begin{equation}\label{deltam}
 \delta_m (X_1,\ldots,X_n)
 =\widehat{\wedge \cG_0}\sum_{(m_1,\ldots ,m_n)}
 \sgn (m_1,\ldots ,m_n) 
 X_1^{q^{m_1-m}} \cdots X_n^{q^{m_n-m}} 
\end{equation}
in 
$\cO_K [[X_1^{1/q^{\infty}},\ldots, X_n^{1/q^{\infty}}]]$, 
where 
\begin{itemize}
\item 
the symbol $\widehat{\wedge \cG_0}\sum$ 
denotes the sum under the additive operation of 
$\widehat{\wedge \cG_0}$, 
\item 
we take the sum over $n$-tuples 
$(m_1,\ldots ,m_n)$ of integers, which can be negative, 
such that 
$m_1 + \cdots + m_n =n(n-1)/2$ and 
$m_i \not\equiv m_j \mod n$ for $i \neq j$, 
\item 
$\sgn (m_1,\ldots ,m_n)$ 
is the sign of the permutation on $\bZ/n\bZ$ 
defined by 
$i \mapsto m_{i+1}$. 
\end{itemize}
We put $\delta=\lim_{m \to \infty} \delta_m^{q^m} \in 
\mathcal{O}_{\mathbf{C}}[[X_1^{1/q^{\infty}},\ldots, X_n^{1/q^{\infty}}]]$. 
For $l \geq 1$, 
we put $\delta^{q^{-l}}=\lim_{m \to \infty} \delta_m^{q^{m-l}}$. 

\begin{thm}\label{thm:Astr}
Let 
$\sigma \in \mathrm{Gal}
 (\widehat{K}^{\mathrm{ab}}/\widehat{K}^{\mathrm{ur}})$. 
We put 
$A^{\sigma}=A 
 \widehat{\otimes}_{\mathcal{O}_{\widehat{K}^{\mathrm{ab}}}, \sigma} 
 \mathcal{O}_{\mathbf{C}}$.
Then, we have an isomorphism 
\begin{equation}\label{Astr}
 A^{\sigma} 
 \simeq 
 \mathcal{O}_{\mathbf{C}}[[X_1^{1/q^{\infty}},\ldots,X_n^{1/q^{\infty}}]]/
 (\delta(X_1,\ldots,X_n)^{q^{-m}} - \sigma(t^{q^{-m}}))_{m \geq 0}^-, 
\end{equation}
where $(\delta(X_1,\ldots,X_n)^{q^{-m}} - \sigma(t^{q^{-m}}))_{m \geq 0}^-$ 
is the closure of the ideal generated by 
$\delta(X_1,\ldots,X_n)^{q^{-m}} - \sigma(t^{q^{-m}})$ for $m \geq 0$. 
\end{thm}
\begin{proof}
This follows from \cite[(2.8)]{WeSemi}. 
See the proof of \cite[Theorem 2.10.3]{BWMax} 
for the description of $\delta$ 
(cf.~\cite[Theorem 6.4.1]{ScWeMpd}). 
\end{proof}
\begin{rem}\label{rem:univbasis}
Let $\mathcal{G}^{\mathrm{univ}}$ be the universal formal 
$\mathcal{O}_K$-module over $A_0$. 
Let $U^{(m)}_1, \ldots, U^{(m)}_n \in A_m$ 
denote the universal Drinfeld 
basis for $\mathcal{G}^{\mathrm{univ}}[\mathfrak{p}^m]$ over 
$A_m$.
The elements $U^{(m)}_1, \ldots, U_n^{(m)}$ 
generate the maximal ideal of 
the regular local ring $A_m$. 
The isomorphism 
\eqref{Astr} sends $X_i$ to the limit 
$\lim_{m \to \infty} (U^{(m)}_i)^{q^{(m-1)n}} \in A$. 
\end{rem}

For 
$\sigma \in \Gal (\widehat{K}^{\mathrm{ab}}/\widehat{K}^{\mathrm{ur}})$, 
let $\mathcal{M}_{\infty,\bar{\eta},\sigma}$ 
be the base change of $\cM_{\infty,\eta}$ 
by $\bar{\eta} \to \eta \xrightarrow{\sigma} \eta$. 
For 
$\sigma \in \Gal (\widehat{K}^{\mathrm{ab}}/\widehat{K}^{\mathrm{ur}})$ 
and 
$\alpha =a_K (\sigma) \in \mathcal{O}_K^{\times}$, 
we write $A^{\alpha}$ for $A^{\sigma}$ and 
$\mathcal{M}^{(0)}_{\infty,\bar{\eta},\alpha}$ for 
$\mathcal{M}^{(0)}_{\infty,\bar{\eta},\sigma}$. 
We put 
\begin{equation}\label{Mdec}
 \bfM_{\infty,\mathcal{O}_{\mathbf{C}}}^{(0)} 
 =\coprod_{\alpha \in \mathcal{O}_K^{\times}} 
 \Spf A^{\alpha}, \quad 
 \mathcal{M}^{(0)}_{\infty,\bar{\eta}} 
 =\coprod_{\alpha \in \mathcal{O}_K^{\times}} 
 \mathcal{M}_{\infty,\bar{\eta},\alpha}. 
\end{equation}
Then $\mathcal{M}^{(0)}_{\infty,\bar{\eta}}$ 
is the generic fiber of 
$\bfM_{\infty,\mathcal{O}_{\mathbf{C}}}^{(0)}$, 
and 
$\mathcal{M}^{(0)}_{\infty,\bar{\eta}} (\bfC) = \cM_{\infty}(\bfC)$ as sets. 

Let $+_{\wh{\cG_0}}$ and 
$+_{\wh{\wedge \cG_0}}$ be the additive operations for 
$\wh{\cG_0}$ and $\wh{\wedge \cG_0}$ respectively. 

\begin{lem}\label{appsum}
{\rm 1.} 
We have $X_1 +_{\wh{\cG_0}} X_2 \equiv X_1 + X_2$ 
modulo terms of total degree $q^n$. \\ 
{\rm 2.} 
We have $X_1 +_{\wh{\wedge \cG_0}} X_2 \equiv X_1 + X_2$ 
modulo terms of total degree $q$. 
\end{lem}
\begin{proof}
This follows from the descriptions of the logarithms of 
$\wh{\cG_0}$ and $\wh{\wedge \cG_0}$ (cf.~\cite[Lemma 5.2.1]{WeSemi}). 
\end{proof}

Let $\bom{X}\!_i$ be $(X_i^{q^{-j}} )_{j \geq 0}$ for $1 \leq i \leq n$. 
We write 
$\delta(\bom{X}\!_1, \ldots,\bom{X}\!_n)$ for 
the $q$-th power compatible system 
$(\delta(X_1,\ldots,X_n)^{q^{-m}} )_{m \geq 0}$. 

For $q$-th power compatible systems 
$\bom{X}=(X^{q^{-j}} )_{j \geq 0}$ and 
$\bom{Y}=(Y^{q^{-j}} )_{j \geq 0}$ that take values in $\cO_{\bfC}$, 
we define $q$-th power compatible systems 
$\bom{X} +\bom{Y}$, $\bom{X} -\bom{Y}$ and $\bom{X} \bom{Y}$ 
by the requirement that their $j$-th components for $j \geq 0$ are 
\[
 \lim_{m \to \infty} (X^{q^{-m}} +Y^{q^{-m}})^{q^{m-j}}, \quad 
 \lim_{m \to \infty} (X^{q^{-m}} -Y^{q^{-m}})^{q^{m-j}}, \quad 
 \textrm{and} \quad 
 X^{q^{-j}} Y^{q^{-j}}
\]
respectively. 
For such $\bom{X}=(X^{q^{-j}} )_{j \geq 0}$, 
we put $v(\bom{X}) = v(X)$. 
We put 
\begin{equation}\label{eq:del0'def}
 \delta'_0 (\bom{X}\!_1, \ldots,\bom{X}\!_n) 
 =\sum_{(m_1,\ldots ,m_n)} 
 \sgn (m_1,\ldots ,m_n) 
 \bom{X}\!_1^{\,q^{m_1}} \cdots \bom{X}\!_n^{\,q^{m_n}} , 
\end{equation}
where 
we take the sum in the above sense and 
the index set is the same as \eqref{deltam}. 
\begin{lem}\label{mixed}
Assume that $n \geq 2$ and 
$v(\bom{X}\!_i) \geq (n q^{i-1}(q-1))^{-1}$ for $1 \leq i \leq n$. 
Then, we have 
\[
 \delta(\bom{X}\!_1, \ldots,\bom{X}\!_n) \equiv 
 \delta'_0 (\bom{X}\!_1,\ldots,\bom{X}\!_n) \mod\!_>\, 
 \frac{1}{n}+\frac{1}{q-1}. 
\]
\end{lem}
\begin{proof}
It suffices to work on $X_1, \ldots,X_n$. 
We write 
\[
 \delta_0=\wh{\wedge \cG_0} 
 \sum_{i=1}^{\infty} \left( 
 \iota_i\prod_{j=1}^n X_j^{r_{i,j}} \right)
\]
with $\iota_i \in \{ \pm 1 \}$ and positive rational numbers 
$r_{i,j} \in \mathbb{Z}[q^{-1}]$ as in \eqref{deltam}. 
We set 
$f_i=\prod_{j=1}^n X_j^{r_{i,j}}$ for all $i \geq 1$. 

We show that $v (f_i) \geq 1/(q-1)$ for all $i \geq 1$. 
For this, we may assume that 
$v(X_i) = (n q^{i-1}(q-1))^{-1}$ for all $1 \leq i \leq n$. 
By this equality and the definition of $\delta_0$, 
the minimality of $v (f_i)$ is achived if 
$r_{i,j}=q^{j-1}$ for $1 \leq j \leq n$, 
in which case $v (f_i) = 1/(q-1)$. 

We note that 
$\delta_m^{q^m} \equiv \delta_0 \mod\!_>\, 1$. 
Hence we obtain 
\[
 \delta_m^{q^m} \equiv \delta_0 \equiv 
 \delta'_0 \mod\!_>\, 
 \frac{1}{n}+\frac{1}{q-1} 
\]
by Lemma \ref{appsum} and $v (f_i) \geq 1/(q-1)$ for $i \geq 1$. 
The claim follows from this. 
\end{proof}

\subsection{Group action on the formal model}\label{gpfm}
We define a group action on the formal scheme 
$\bfM_{\infty,\mathcal{O}_{\mathbf{C}}}^{(0)}$, 
which is compatible with usual group actions on 
Lubin-Tate spaces with finite level (cf.~\cite[2.11]{BWMax}). 
We put 
\[
 G=\mathit{GL}_n(K) \times D^{\times} \times W_K.
\] 
Let $G^0$ denote the kernel of the following homomorphism:
\[
 G \to \mathbb{Z} ;\ (g,d,\sigma) \mapsto v 
 \bigl( \det(g)^{-1} 
 \mathrm{Nrd}_{D/K}(d)
 \mathrm{Art}^{-1}_K(\sigma) \bigr). 
\]
Then, the formal scheme 
$\bfM_{\infty, \mathcal{O}_{\mathbf{C}}}^{(0)}$ 
admits a right action of $G^0$. 
We write down the action. 
In the sequel, we use the following notation: 
\begin{quote}
For $a \in \mu_{q^n -1} (K_n) \cup \{0\}$, 
let $a^{q^{-m}}$ denote the 
$q^m$-th root of $a$ in $\mu_{q^n -1} (K_n) \cup \{0\}$ 
for a positive integer $m$, and 
we simply write $a$ also for 
$q$-th power compatible system $(a^{q^{-m}})_{m \geq 0}$. 
\end{quote}

For $q$-th power compatible systems 
$\bom{X}=(X^{q^{-j}} )_{j \geq 0}$ and 
$\bom{Y}=(Y^{q^{-j}} )_{j \geq 0}$ that take values in $\cO_{\bfC}$, 
we define $q$-th power compatible system 
$\bom{X} +_{\wh{\cG_0}} \bom{Y}$ 
by the requirement that their $j$-th components for $j \geq 0$ are 
$\lim_{m \to \infty} (X^{q^{-m}} +_{\wh{\cG_0}} Y^{q^{-m}})^{q^{m-j}}$. 
The symbol $\wh{\cG_0} \sum$ 
denotes this summation for $q$-th power compatible systems. 

First, we define 
a left action of $\mathit{GL}_n(K) \times D^{\times}$
on the ring 
$B_n=\mathcal{O}_{\mathbf{C}} 
 [[X_1^{1/q^{\infty}},\ldots,X_n^{1/q^{\infty}}]]$. 
For 
$a=\sum_{j=l}^{\infty}a_j\varpi^{j} \in K$ with 
$l \in \mathbb{Z}$ and 
$a_j \in \mu_{q-1} (K) \cup \{0\}$, 
we set
\[
 [a] \cdot \bom{X}\!_i = 
 \wh{\cG_0} \sum_{j=l}^{\infty} a_j \bom{X}\!_i^{\,q^{jn}} 
\]
for $1 \leq i \leq n$. 
Let $g \in \mathit{GL}_n(K)$. 
We write 
$g =(a_{i,j})_{1 \leq i,j \leq n}$.
Then, let $g$ act on the ring $B_n$
by 
\begin{equation}\label{gl}
 g^{\ast} \colon B_n \to B_n;\ \bom{X}\!_i \mapsto 
 \wh{\cG_0} \sum_{j=1}^n [a_{j,i}] \cdot \bom{X}\!_j 
 \quad \textrm{for $1 \leq i \leq n$}.
\end{equation}
Let $d \in D^{\times}$. 
We write 
$d^{-1} =\sum_{j=l}^{\infty} d_j \varphi^j \in D^{\times}$ 
with 
$l \in \mathbb{Z}$ 
and $d_j \in \mu_{q^n -1} (K_n) \cup \{0\}$. 
Then, let $d$ act on $B_n$ by 
\begin{equation}\label{divi}
 d^{\ast} \colon B_n \to B_n ;\ 
 \bom{X}\!_i \mapsto 
 \wh{\cG_0} \sum_{j=l}^{\infty } d_j \bom{X}\!_i^{\,q^{j}} 
 \quad \textrm{for $1 \leq i \leq n$}.
\end{equation}
Now, we give a right action of $G^0$ on 
$\bfM_{\infty, \mathcal{O}_{\mathbf{C}}}^{(0)}$ 
using Theorem \ref{thm:Astr}, 
\eqref{gl} and \eqref{divi}. 
Let $(g,d,1) \in G^0$. 
We set 
$\gamma(g,d)=
 \det (g) \mathrm{Nrd}_{D/K}(d)^{-1} \in 
 \mathcal{O}_K^{\times}$. 
We put $\bom{t} =(t^{q^{-m}})_{m \geq 0}$. 
Let $(g,d,1)$ act on 
$\bfM_{\infty, \mathcal{O}_{\mathbf{C}}}^{(0)}$ 
by 
\[
 A^{\alpha} \to A^{\gamma(g,d)^{-1} \alpha };\ 
 \bom{X}\!_i \mapsto (g,d) \cdot \bom{X}\!_i 
 \quad \textrm{for $1 \leq i \leq n$}, 
\]
where $\alpha \in \cO_K^{\times}$. 
This is well-defined, because 
the equation 
\[
 \delta ((g,d) \cdot \bom{X}\!_1 , \ldots, (g,d) \cdot \bom{X}\!_n )=
 \Art_K (\alpha) (\bom{t}) 
\]
is equivalent to 
$\delta (\bom{X}\!_1 , \ldots, \bom{X}\!_n )=
 \Art_K (\gamma(g,d)^{-1} \alpha) (\bom{t})$ 
by \cite[2.9]{WeSemi}. 
Let $(1,\varphi^{-n_{\sigma}},\sigma) \in G^0$ act on 
$\bfM_{\infty, \mathcal{O}_{\mathbf{C}}}^{(0)}$ by 
\[
 A^{\alpha} \to A^{a_K(\sigma)\alpha};\ \bom{X}\!_i \mapsto 
 \bom{X}\!_i, \hspace{1.0em} x \mapsto 
 \sigma (x) \quad 
 \textrm{ for $1 \leq i \leq n$ and $x \in \mathcal{O}_{\mathbf{C}}$}, 
\]
where $\alpha \in \cO_K^{\times}$. 
Thus, we have a right action of $G^0$ on 
$\bfM_{\infty,\mathcal{O}_{\mathbf{C}}}$, 
which induces a right action on 
$\mathcal{M}^{(0)}_{\infty,\bar{\eta}} (\bfC) = \cM_{\infty}(\bfC)$. 

\begin{rem}\label{Ktri}
For $a \in K^{\times}$, the action of $(a,a,1) \in G^{0}$ 
is trivial by the definition. 
\end{rem}

\subsection{CM points}\label{CMpts}
We recall the notion of CM points from \cite[3.1]{BWMax}. 
Let $L$ be a finite extension of $K$ of degree $n$ 
inside $\bfC$. 

\begin{defn}
A deformation $\cG$ of $\cG_0$ over $\cO_{\bfC}$ 
has CM by $L$ if 
there is an isomorphism 
$L \xrightarrow{\sim} \End (\cG) \otimes_{\cO_K} K$ 
as $K$-algebras 
such that the induced map 
$L \to \End (\Lie \cG) \otimes_{\cO_K} K \simeq \bfC$ 
coincides with the natural embedding $L \subset \bfC$. 
\end{defn}

We say that a point of $\cM_{\infty}(\bfC)$ 
has CM by $L$ if the corresponding deformation over $\cO_{\bfC}$ 
has CM by $L$. 

Let $\xi \in \cM_{\infty}(\bfC)$ be a 
point that has CM by $L$. 
Let $(\cG,\phi,\iota)$ be the triple corresponding to $\xi$. 
Then we have embeddings 
$i_{M,\xi} \colon L \to M_n (K)$ 
and 
$i_{D,\xi} \colon L \to D$ 
characterized by the commutative diagrams 
\[
 \xymatrix{
 K^n \ar@{->}^-{\phi}[r] \ar@{->}_-{i_{M,\xi} (a)}[d] & 
 V_{\fp} \cG  
 \ar@{->}^-{V_{\fp} (a)}[d] \\ 
 K^n \ar@{->}^-{\phi}[r] & 
 V_{\fp} \cG  
 }
 \quad \quad \lower20pt\hbox{\textrm{and}} \quad \quad 
 \xymatrix{
 \cG_0 \ar@{->}^-{\iota}[r] \ar@{->}_-{i_{D,\xi} (a)}[d] & 
 \cG \otimes_{\cO_{\bfC}} k^{\rmac}  
 \ar@{->}^-{a \otimes \id }[d] \\ 
 \cG_0 \ar@{->}^-{\iota}[r] & 
 \cG \otimes_{\cO_{\bfC}} k^{\rmac} 
 }
\]
in the category of $K$-vector spaces and 
in the isogeny category of $p$-divisible groups over $k^{\rmac}$ 
for $a \in L$, respectively. 
We put 
$i_{\xi}=(i_{M,\xi} ,i_{D,\xi}) \colon L \to M_n (K) \times D$. 
We put 
\[
 (\iGL_n (K) \times D^{\times})^0 =\{ (g,d) \in 
 \iGL_n (K) \times D^{\times} \mid (g,d,1) \in G^0 \}. 
\]

\begin{lem}\cite[Lemma 3.1.2]{BWMax}\label{CMtrans}
The group $(\iGL_n (K) \times D^{\times})^0$ 
acts transitively on the set of the points of 
$\cM_{\infty}(\bfC)$ 
that have CM by $L$. 
For $\xi \in \cM_{\infty}(\bfC)$ 
that has CM by $L$, 
the stabilizer of $\xi$ in $(\iGL_n (K) \times D^{\times})^0$ 
is $i_{\xi} (L^{\times})$. 
\end{lem}

\section{Reductions of formal models of affinoids}\label{GoodRedAff}
\subsection{Construction of affinoids}\label{ConstAff}
We assume that 
$n \geq 2$ and 
$n$ is prime to $p$. 
We put 
\begin{equation}\label{eq:defnq}
 n_q =\gcd (n,q-1). 
\end{equation} 
For a uniformizer $\varpi$ of $K$, 
we put $L_{\varpi} =K[T]/(T^n -\varpi)$. 
Let $T(K,n)$ be the set of the isomorphism classes of 
totally ramified extensions of 
$K$ of degree $n$. 
\begin{lem}\label{ttbij}
We have the bijection 
\[
 \mu_{(q-1)/n_q} (K) \backslash (\fp_K -\fp_K^2 )/\fp_K^2 
 \to T(K,n); \ 
 \varpi \mapsto L_{\varpi}. 
\]
\end{lem}
\begin{proof}
We can see the well-definedness and the injectivity easily. 
The surjectivity follows from 
\cite[II Proposition 12]{LangANT}. 
\end{proof}

Let $L$ be a 
totally ramified extension of 
$K$ of degree $n$ in $\bfC$. 
By a representation theoretic reason, 
an affinoid which we try to find should be stable under the action of 
a subgroup of $(\iGL_n (K) \times D^{\times})^0$ 
containing $i_{\xi} (L^{\times})$ for some point $\xi$ having CM by $L$. 
See Subsection \ref{subsec:stab} and 
Section \ref{RealLLCLJLC} for the precise situation for the 
stabilizer and the representation theory. 
By this reason, we can expect that 
the desired affinoid is defined near the CM point. 
The coordinate of the Lubin-Tate perfectoid space 
introduced in Subsection \ref{LTpsfm} 
depends on the choice of a uniformizer $\varpi$ of $K$. 
In the following, we pick up a 
coordinate of the Lubin-Tate perfectoid space, 
which is suitable to describe a CM point by $L$, 
by chosing a uniformizer of $K$. 
Then we define an affinoid near the CM point. 
We will see in Subsection \ref{subsec:stab} that 
the affinoid has an appropriate stabilizer. 
A similar situation should occur for a more general class of representations 
(\cf Remark \ref{rem:compari}). 

We take a uniformizer $\varpi_L$ of $K$ such that $L \simeq L_{\varpi_L}$. 
Further, we take $\varphi_L \in L$ such that $\varphi_L^n =\varpi_L$. 
By the $\cO_K$-algebra embedding 
$\cO_L \to \cO_D$ defined by $\varphi_L \mapsto \varphi$, 
we view $\cG_0$ as a formal $\cO_L$-module of height $1$. 
Let $\cG^L$ 
be a lift of $\cG_0$ to $\cO_{\wh{L}^{\rmur}}$ 
as formal $\cO_{L}$-modules. 
We take a compatible system $\{t_{L,m}\}_{m \geq 1}$ 
such that 
\[
 t_{L,m} \in \bfC \quad (m \geq 1), \quad 
 t_{L,1} \neq 0, \quad 
 [\varphi_{L}]_{\cG^L}(t_{L,1})=0, \quad 
 [\varphi_{L}]_{\cG^L}(t_{L,m})=t_{L,m-1} \quad (m \geq 2).
\]
We apply results in Section \ref{LTps} for $\varpi =\varpi_L$. 
We put 
\[
 \varphi_{M,L} = 
 \begin{pmatrix}
 \bm{0} & I_{n-1} \\
 \varpi_L & \bm{0} \\
 \end{pmatrix}
 \in M_n (K) 
\]
and $\varphi_{D,L} = \varphi \in D$. 

For $\xi \in \cM_{\infty}(\bfC)=\cM^{(0)}_{\infty,\ol{\eta}}(\bfC)$, 
let $(\bom{\xi}_1,\ldots,\bom{\xi}_n)$ be the coordinate of $\xi$ 
with respect to $(\bom{X}\!_1,\ldots,\bom{X}\!_n)$, 
where $\bom{\xi}_i =(\xi_{i}^{q^{-j}})_{j \geq 0}$ for 
$1 \leq i \leq n$ (\cf Remark \ref{rem:univbasis}). 

\begin{lem}\label{exxi}
There exists 
$\xi_L \in \cM_{\infty}(\bfC)$ 
such that 
\begin{equation}\label{xilim}
 \xi_{L,i}^{q^{-j}} =\lim_{m \to \infty}  
 t_{L,m}^{q^{m-i-j}} \in \mathcal{O}_{\mathbf{C}} 
\end{equation}
for $1 \leq i \leq n$ and $j \geq 0$. 
Further, we have the following: 
\begin{enumerate}
\item 
The point $\xi_L$ has CM by $L$. 
\item 
We have 
$i_{\xi_L} (\varphi_{L}) = 
 (\varphi_{M,L} ,\varphi_{D,L} ) \in M_n (K) \times D$. 
\item 
We have $\bom{\xi}_{L,i} = \bom{\xi}_{L,i+1}^q$ 
for $1 \leq i \leq n-1$. 
\item 
We have $v(\xi_{L,i})=1/(n q^{i-1}(q-1))$ for $1 \leq i \leq n$. 
\end{enumerate}
\end{lem}
\begin{proof}
For $m \geq 1$, 
the ordered set 
$\{ t_{L,j} \}_{n(m-1)+1 \leq j \leq nm}$ is an ordered basis of 
$\cG^L [\fp_{K}^m]$ over $\cO_K/\fp_{K}^m$. 
They form a compatible system of ordered base of 
$\cG^L [\fp_{K}^m]$ for $m \geq 1$ with respect to 
multiplication by $\varpi_L$. 
This system gives a point 
$\xi_L \in \cM_{\infty}(\bfC)$, 
whose coordinates are given as \eqref{xilim} 
by the construction of the parameter $(X_1,\ldots,X_n)$. 
The properties follow from the construction. 
\end{proof}

We take $\xi_L$ as in Lemma \ref{exxi}. 
If 
$\xi_L \in \mathcal{M}_{\infty,\bar{\eta},\sigma}(\bfC) \subset \cM_{\infty}(\bfC)$ (\cf \eqref{Mdec}), 
then we replace the choice of \eqref{tmcho} 
by the action of $\sigma$. 
Then we have 
$\delta (\bom{\xi}_1,\ldots,\bom{\xi}_n) =\bom{t}$ 
and 
$\xi_L \in \mathcal{M}_{\infty,\bar{\eta}}(\bfC) \subset \cM_{\infty}(\bfC)$. 

Let 
$\cD_{\bfC}^{n,\mathrm{perf}}$ 
be the generic fiber of 
$\Spf \mathcal{O}_{\mathbf{C}}
 [[X_1^{1/q^{\infty}},\ldots,X_n^{1/q^{\infty}}]]$. 
We consider 
$\cM_{\infty,\ol{\eta}}$ as a subspace of 
$\cD_{\bfC}^{n,\mathrm{perf}}$ 
by \eqref{Astr}. 
We define 
$\cX^{L} \subset \cM_{\infty,\ol{\eta}}$ by 
\begin{equation}\label{affc}
 v \biggl( \frac{X_n}{\xi_{L,n}} -1 
 \biggr)  \geq \frac{1}{2n q^{n-1}} \quad 
 \textrm{and} \quad 
 v \biggl( \frac{X_i}{\xi_{L,i}} 
 -\frac{X_{i+1}}{\xi_{L,i+1}} \biggr)  \geq \frac{1}{2n q^i} \quad 
 \textrm{for} \ 1 \leq i \leq n-1. 
\end{equation}
We define $\mathcal{B}^{L} \subset \cD_{\bfC}^{n,\mathrm{perf}}$ 
by the same condition \eqref{affc}. 
We note that the condition 
\eqref{affc} implies 
that 
\begin{equation}\label{eq:v(Xi)}
 v(X_i)=\frac{1}{n q^{i-1}(q-1)}, \quad  
v \biggl( \frac{X_i}{\xi_{L,i}} -1 \biggr) \geq \frac{1}{2n q^{n-1}} 
\end{equation}
for $1 \leq i \leq n$. 

\subsection{Formal models of affinoids}\label{RedAff}
Let $\bom{1}$ be the $q$-th power compatible system 
whose all components are $1$. 
Let $(\bom{X}\!_1,\ldots,\bom{X}\!_n)$ be the coordinate of $\cB^{L}$. 
We put 
$h (\bom{X}\!_1 ,\ldots ,\bom{X}\!_n )=\prod_{i=1}^n \bom{X}\!_i^{\,q^{i-1}}$. 
Further, we put 
\begin{align}
 f(\bom{X}\!_1, \ldots, \bom{X}\!_n) &=\bom{1}-\frac{\delta (\bom{X}\!_1, \ldots, \bom{X}\!_n)}{h (\bom{X}\!_1 ,\ldots ,\bom{X}\!_n )}, \label{fdef} \\ 
 f_0 (\bom{X}\!_1, \ldots, \bom{X}\!_n) &= 
 \sum_{i=1}^{n-1}
 \biggr( \frac{\bom{X}\!_i}{\bom{X}\!_{i+1}} \biggr)^{q^{i-1}(q-1)}
 +\biggl( \frac{\bom{X}\!_n^{\,q^n}}{\bom{X}\!_1} \biggr)^{\frac{q-1}{q}} \label{f0def}. 
\end{align}
We simply write $f(\bom{X})$ for $f(\bom{X}\!_1, \ldots, \bom{X}\!_n)$, and 
$f(\bom{\xi}_L)$ for $f(\bom{\xi}_{L,1}, \ldots, \bom{\xi}_{L,n})$. 
We will use the similar notations also for other functions. 
We put 
\begin{equation}\label{Zdef}
 \bom{Z} =f_0(\bom{X}) -f_0(\bom{\xi}_L). 
\end{equation}

\begin{lem}\label{fZapp}
We have 
\[
 f(\bom{X}) \equiv f_0 (\bom{X}) \mod\!_>\, \frac{q-1}{nq} 
 \quad \textrm{and} \quad 
 \bom{Z} \equiv f(\bom{X}) -f(\bom{\xi}_L) \mod\!_>\, \frac{1}{n} 
\]
on $\cB^{L}$. 
\end{lem}
\begin{proof}
We put 
\[
 f_1 (\bom{X}\!_1, \ldots, \bom{X}\!_n) =\sum_{i=1}^{n-3}
 \sum_{j=i+2}^{n-1}
 \biggl( \frac{\bom{X}\!_i}{\bom{X}\!_{i+1}} \biggr)^{q^{i-1}(q-1)}
 \biggl( \frac{\bom{X}\!_j}{\bom{X}\!_{j+1}} \biggr)^{q^{j-1}(q-1)}
 +\biggl( \frac{\bom{X}\!_n^{\,q^n}}{\bom{X}\!_1} \biggr)^{\frac{q-1}{q}} 
 \sum_{i=1}^{n-3}
 \biggl( \frac{\bom{X}\!_{i+1}}{\bom{X}\!_{i+2}} \biggr)^{q^i(q-1)}. 
\]
By the definition of $\delta_0'$ and \eqref{eq:v(Xi)}, 
we see that 
\[
 \frac{\delta_0' (\bom{X})}{h (\bom{X})}  \equiv 
 1 - f_0 (\bom{X}) + f_1 (\bom{X}) \mod\!_>\, \frac{1}{n}, 
\]
where the term $1$ comes from 
the index $(0, 1, \ldots , n-1)$ in \eqref{eq:del0'def}, 
the term 
$f_0 (\bom{X})$ comes from 
the indexes obtained from 
$(0, 1, \ldots , n-1)$ by adding $1$ to $i$-th component 
and substracting $1$ from $(i+1)$-th component for $1 \leq i \leq n$, 
(where $(n+1)$-th means $1$-st,) 
the term 
$f_1 (\bom{X})$ comes from 
the indexes obtained from 
$(0, 1, \ldots , n-1)$ by similar changes twice, 
and the terms coming from other indexes have higher valuations. 
Hence we have 
\[
 f(\bom{X}) \equiv 
 \bom{1}-\frac{\delta_0' (\bom{X})}{h (\bom{X})} \equiv 
 f_0 (\bom{X}) -f_1 (\bom{X}) \mod\!_>\, \frac{1}{n}
\]
using Lemma \ref{mixed}. 
The claims follow from this, 
because 
\[
 v(f_1 (\bom{X})) \geq \frac{2(q-1)}{nq} \quad \textrm{and} \quad 
 v \bigl( f_1 (\bom{X}) -f_1 (\bom{\xi}_L) \bigr) > \frac{2(q-1)}{nq} 
\]
by \eqref{eq:v(Xi)}. 
\end{proof}

We put $\bom{x}_i =\bom{X}\!_i /\bom{\xi}_{L,i}$ for 
$1 \leq i \leq n$. 
We set 
\begin{equation}\label{Ydef}
 \biggl( \frac{\bom{x}_i}{\bom{x}_{i+1}} 
 \biggr)^{q^i (q-1)} 
 =\bom{1} + \bom{Y}\!_i \quad \textrm{with} 
 \quad v(\bom{Y}\!_i) \geq \frac{1}{2n} 
\end{equation}
for $1 \leq i \leq n-1$.

We define a subaffinoid 
$\mathcal{B}'^{L} \subset \mathcal{B}^{L}$ 
by $v(\bom{Z}) \geq 1/n$. 
We put $\bom{\eta}_L =\bom{\xi}_{L,1}^{q-1}$. 
Note that $v(\bom{\eta}_L)=1/n$. 
We choose a square root $\bom{\eta}_L^{1/2}$ 
of $\bom{\eta}_L$. 
We set 
\begin{equation}\label{yzdef}
\begin{split}
 &\bom{Y}\!_i =\bom{\eta}_L^{1/2} \bom{y}_i 
 \quad \textrm{with} \quad \bom{y}_i =(y_i^{q^{-j}})_{j \geq 0} 
 \quad \textrm{for $1 \leq i \leq n-1$}, \\ 
 &\bom{Z} =\bom{\eta}_L \bom{z} \quad \textrm{with} \quad 
 \bom{z} =(z^{q^{-j}})_{j \geq 0} 
\end{split}
\end{equation}
on $\mathcal{B}'^{L}$. 
Let $\cB(\bom{y}, \bom{z})$ be the 
generic fiber of 
$\Spf \cO_{\bfC} \langle y_1^{1/q^{\infty}}, \ldots ,
 y_{n-1}^{1/q^{\infty}}, z^{1/q^{\infty}} \rangle$. 
The parameters $\bom{y}_1, \ldots ,\bom{y}_{n-1}, \bom{z}$ give the morphism 
$\Theta \colon \cB'^L \to \cB(\bom{y}, \bom{z})$. 
\begin{lem}\label{BBisom}
The morphism $\Theta$ is an isomorphism. 
\end{lem}
\begin{proof}
We will construct the inverse morphism of $\Theta$. 
By solving \eqref{Ydef} iteratively, 
we can write $\bom{x}_i /\bom{x}_{i+1}$ as 
a $q$-th power compatible system of analytic functions on 
$\cB (\bom{y}, \bom{z})$. 
By \eqref{f0def}, \eqref{Zdef} and $\bom{\eta}_L =\bom{\xi}_{L,1}^{q-1}$, 
we have 
\begin{align}
 \bom{Z}^q &= 
 \bom{\eta}_L^{q-1} \Biggl( \sum_{i=1}^{n-1} 
 \Biggl( \biggl( \frac{\bom{x}_i}{\bom{x}_{i+1}}\biggr)^{q^i(q-1)} 
 -\bom{1} \Biggr) + 
 \biggl( \frac{\bom{x}_n^{q^n}}{\bom{x}_1}\biggr)^{q-1} 
 -\bom{1} \Biggr) \label{eq:Z^q} \\ 
 &=\bom{\eta}_L^{q-1} \Biggl( \sum_{i=1}^{n-1} \bom{Y}\!_i 
 +\bom{x}_{n}^{(q-1)(q^n-1)} \prod_{i=1}^{n-1} 
 (\bom{1}+\bom{Y}\!_i )^{-q^{-i}}  -\bom{1} \Biggr), \notag 
\end{align}
where we use \eqref{Ydef} at the last equality. 
Solving the above equation, 
we can write $\bom{x}_n$ as 
a $q$-th power compatible system of analytic functions on 
$\cB(\bom{y}, \bom{z})$ with integral coefficients. 
Hence, we have the inverse morphism of $\Theta$. 
\end{proof}
We put 
\[
 \delta_B (y_1, \ldots ,y_{n-1}, z) = 
 (\delta|_{\cB_L'} )\circ \Theta^{-1}
\]
equipped with its $q^j$-th root 
$\delta_B^{q^{-j}}$ for $j \geq 0$. 
We put 
\[
 \fX^L =\Spf \cO_{\bfC} \langle y_1^{1/q^{\infty}}, \ldots ,
 y_{n-1}^{1/q^{\infty}}, z^{1/q^{\infty}} \rangle
 /(\delta_B^{q^{-j}})_{j \geq 0}. 
\]

\begin{thm}\label{redmod}
The formal scheme $\fX^L$ is 
a formal model of $\cX^L$, 
and the special fiber of $\fX^L$ is isomorphic to 
the perfection of the affine smooth variety defined by 
\[
 z^q -z =\sum_{1 \leq i \leq j \leq n-1} y_i y_j \quad 
 \textrm{in} \ \bA_{k^{\rmac}}^n. 
\]
\end{thm}
\begin{proof}
Let $(\bom{X}\!_1,\ldots,\bom{X}\!_n)$ be the coordinate of $\cB^{L}$. 
By Lemma \ref{fZapp}, we have 
\begin{equation}\label{vfZ}
 v(f(\bom{X})) \geq \frac{q-1}{nq} \quad \textrm{and} 
 \quad 
 v(\bom{Z}) > \frac{q-1}{nq}. 
\end{equation}
We put 
\begin{equation}\label{Fdef}
 F(\bom{X})= 
 \frac{\bom{1}-f(\bom{\xi}_L)}{\bom{1}-f(\bom{X})} 
 - (\bom{1}+\bom{Z}). 
\end{equation}
Then we have $v(F(\bom{X})) >1/n$ by 
Lemma \ref{fZapp} and \eqref{vfZ}. 
We have 
\begin{equation}\label{hXX}
 h(\bom{X})^{q-1}= 
 \biggl( \frac{\bom{X}\!_n^{\,q^n}}{\bom{X}\!_1} \biggr)
 \prod_{i=1}^{n-1} \biggl( 
 \frac{\bom{X}\!_i}{\bom{X}\!_{i+1}} \biggr)^{q^i} . 
\end{equation}
We have 
$h(\bom{\xi}_L)^{(q-1)^2} = \bom{\eta}_L^{n(q-1)}$ 
by Lemma \ref{exxi}.(iii) and \eqref{hXX}. Hence, we have 
\begin{equation}\label{XhhY}
 \biggl( \frac{\bom{X}\!_n^{\,q^n}}{\bom{X}\!_1} \biggr)^{q-1}
 = 
 \bom{\eta}_L^{q-1} 
 \biggl( \frac{h(\bom{X})}{h(\bom{\xi}_L)} 
 \biggr)^{(q-1)^2} 
 \prod_{i=1}^{n-1} (\bom{1}+\bom{Y}\!_i)^{-1} 
\end{equation}
by \eqref{hXX}. 
Then $\delta (\bom{X}) =\delta (\bom{\xi}_L)$ is equivalent to 
\begin{equation}\label{XZFY}
 \biggl( \frac{\bom{X}\!_n^{\,q^n}}{\bom{X}\!_1} \biggr)^{q-1}
 = 
 \bom{\eta}_L^{q-1} 
 \bigl( \bom{1} +\bom{Z} +F(\bom{X} )  
 \bigr)^{(q-1)^2} 
 \prod_{i=1}^{n-1} (\bom{1}+\bom{Y}\!_i)^{-1}
\end{equation}
by \eqref{fdef}, \eqref{Fdef} and \eqref{XhhY}. 
We put 
\[
 R(\bom{X})= 
 \sum_{i=1}^{n-1} \bom{Y}\!_i + 
 \bigl( \bom{1} +\bom{Z} +F(\bom{X} )  
 \bigr)^{(q-1)^2} 
 \prod_{i=1}^{n-1} (\bom{1}+\bom{Y}\!_i)^{-1} 
 -\bom{1} 
 -\sum_{1 \leq i \leq j \leq n-1}\bom{Y}\!_i \bom{Y}\!_j 
 -\bom{Z}.
\]
Then we have 
\begin{align*}
 R(\bom{X}) &\equiv 
 \sum_{i=1}^{n-1} \bom{Y}\!_i + 
 \bigl( \bom{1} +\bom{Z} 
 \bigr) 
 \prod_{i=1}^{n-1} (\bom{1}-\bom{Y}\!_i +\bom{Y}\!_i^{\,2} ) 
 -\bom{1} 
 -\sum_{1 \leq i \leq j \leq n-1}\bom{Y}\!_i \bom{Y}\!_j 
 -\bom{Z} \equiv 0 \mod\!_>\, \frac{1}{n} 
\end{align*}
by 
\eqref{Ydef} and \eqref{vfZ} and $v(F(\bom{X})) >1/n$. 
Hence we have $v(R(\bom{X}))>1/n$. 
Equation 
\eqref{XZFY} is equivalent to 
\begin{equation}\label{XnYZ}
 f_0 (\bom{X})^q = 
 \bom{\eta}_L^{q-1} \Biggl( \bom{n} 
 +\sum_{1 \leq i \leq j \leq n-1}\bom{Y}\!_i \bom{Y}\!_j 
 +\bom{Z} +R (\bom{X}) \Biggr) , 
\end{equation}
where $\bom{n}$ is $n$ times sum of $\bom{1}$. 
Further, \eqref{XnYZ} is equivalent to 
\begin{equation}\label{ZqYZ}
 \bom{Z}^q = 
 \bom{\eta}_L^{q-1} \Biggl( 
 \sum_{1 \leq i \leq j \leq n-1}\bom{Y}\!_i \bom{Y}\!_j + \bom{Z} 
 + R (\bom{X}) \Biggr). 
\end{equation}
As a result, 
$\delta (\bom{X}) =\delta (\bom{\xi}_L)$ is equivalent to 
\eqref{ZqYZ} on $\cB^L$. 
By Lemma \ref{fZapp} and \eqref{ZqYZ}, 
we have $v(\bom{Z}) \geq 1/n$ on $\cX^L$. 
This implies $\cX^L \subset \cB'^L$. 
Then, we have the first claim by Lemma \ref{BBisom}, 
since the ring of integral elements for $\cX^L$ 
is isomorphic to the 
completion of the integral closure of 
\[
 \cO_{\bfC} \langle y_1^{1/q^{\infty}}, \ldots ,
 y_{n-1}^{1/q^{\infty}}, z^{1/q^{\infty}} \rangle
 /(\delta_B^{q^{-j}})_{j \geq 0}
\] 
under the isomorphism induced by $\Theta$. 
Substituting \eqref{yzdef} to \eqref{ZqYZ} 
and dividing it by $\bom{\eta}_L^{q}$, 
we obtain 
\begin{equation}\label{eqonB}
 \bom{z}^q -\bom{z} = 
 \sum_{1 \leq i \leq j \leq n-1} \bom{y}_i \bom{y}_j 
 +S(\bom{y}, \bom{z}), 
\end{equation}
where $S(\bom{y}, \bom{z})$ is the $q$-th power 
compatible system of analytic functions on 
$\cB(\bom{y}, \bom{z})$ 
that corresponds to $\bom{\eta}_L^{-1} R (\bom{X})$ 
under $\Theta$. 
We have $v(S(\bom{y}, \bom{z})) >0$. 
Hence, we have the second claim. 
\end{proof}

\begin{rem}\label{rem:compari}
Tokimoto generalizes the 
construction of affinoids to higher depth cases in 
\cite{TokLTs}. 
The condition \eqref{affc} is equivalent to 
the condition for $\cX_1$ in \cite{TokLTs} 
as mentioned in the introduction of \cite{TokLTs}: 
The condition for $\cX_1$ in \cite{TokLTs} is 
\begin{equation}\label{eq:X1cond}
 v \biggl( \frac{X_i}{\xi_{L,i}} -1 
 \biggr)  \geq \frac{1}{2n q^{i-1}} 
 \quad 
 \textrm{for} \ 1 \leq i \leq n \quad 
 \textrm{and} \quad 
 v \Biggl( \sum_{i=1}^n (X_i - \xi_{L,i})^{q^{i-1}} 
 \Biggr) \geq \frac{q}{n(q-1)} 
\end{equation}
in our notations. 
The condition \eqref{eq:X1cond} implies \eqref{affc} clearly. 
Assume the condition \eqref{affc}. 
Then we have $v(\bom{Z}) \geq 1/n$ as in the proof of Theorem \ref{redmod}. 
By \eqref{affc}, we have 
\[
 v \Biggl( \biggl( \frac{\bom{x}_i}{\bom{x}_{i+1}} \biggr)^{q^i(q-1)} 
 -\bom{1} \Biggr) \geq \frac{1}{2n} 
 \quad \textrm{for} \quad 1 \leq i \leq n-1 \quad 
 \textrm{and} \quad  
 v ( \bom{x}_n^{q^n} -\bom{1} ) \geq \frac{q}{2n}. 
\]
By this, \eqref{eq:Z^q} and $v(\bom{Z}) \geq 1/n$, 
we have 
\[
 v \biggl( \frac{X_1}{\xi_{L,1}} -1 
 \biggr)  \geq \frac{1}{2n} . 
\]
We can see that this implies \eqref{eq:X1cond} 
using \eqref{affc}, \eqref{eq:Z^q} and $v(\bom{Z}) \geq 1/n$ again. 
\end{rem}

\section{Group action on the reductions}\label{GpX}
\subsection{Preliminary}
Let $B$ be the subring consisting of 
upper triangular matrices of $\textit{M}_n(k)$. 
Let $\mathfrak{I} \subset 
\textit{M}_n(\mathcal{O}_K)$ be 
the inverse image of $B$
under the reduction map
$\textit{M}_n(\mathcal{O}_K) 
\to \textit{M}_n(k)$. 

\begin{lem}\label{vgd}
Let $(g,d,1) \in G^0$. 
We take the integer $l$ such that 
$d \varphi_{L,D}^{-l} \in \cO_D^{\times}$. 
Let $(\bom{X}\!_1 ,\ldots ,\bom{X}\!_n )$ be the coordinate 
of $\cX^L$. 
Assume $v((g,d) \cdot \bom{X}\!_i) =v(\bom{X}\!_i)$ 
for $1 \leq i \leq n$ at some point of $\cX^L$. 
Then we have 
$(g,d ) \in (\varphi_{M,L},\varphi_{D,L})^l (\fI^{\times} \times \cO_D^{\times} )$. 
\end{lem}
\begin{proof}
Write 
$g=(a_{i,j})_{1 \leq i,j \leq n} \in \iGL_n (K)$. 
By the condition $v((g,d) \cdot \bom{X}\!_i) =v(\bom{X}\!_i)$ 
and the definition of the action of $(g,d,1)$, 
we have 
\begin{equation}\label{vagd}
 \min_{1 \leq j \leq n} 
 \left\{\frac{q^{nv(a_{j,i})}}{nq^{j+l-1}(q-1)}\right\} 
 =\frac{1}{nq^{i-1}(q-1)}
\end{equation}
for $1 \leq i \leq n$. 
For $1 \leq i,j \leq n$, 
the equality $nv(a_{j,i})=j-i+l$ 
implies $j \equiv i-l \pmod{n}$
and 
$v(a_{j,i})=(j-i+l)/n$. 
Hence, we have the claim
by \eqref{vagd}. 
\end{proof}

\subsection{$\mathit{GL}_n$-action}
We put 
\[
 \iGL_n (K)^0 =\{ g \in \iGL_n (K) \mid v(\det (g))=0 \}. 
\] 
Let $\mathfrak{P}$ be the Jacobson radical of 
the order $\mathfrak{I}$. 
We put 
\[
 U_{\mathfrak{I}}^1=1+\mathfrak{P}^1 
 \quad \textrm{and} \quad 
 U_{\mathfrak{I}}^{1, \det =1}=\{g \in U_{\mathfrak{I}}^1 \mid \det (g)=1\}. 
\] 

\begin{prop}\label{ga}
We set 
$r_L (g) =\overline{\mathrm{tr}\ (\varphi_{M,L}^{-1}(g-1))}$ 
for $g \in U_{\mathfrak{I}}^1$. 
Then, $U_{\mathfrak{I}}^{1, \det =1}$ 
is the stabilizer in $\iGL_n (K)^0$ of 
$\mathcal{X}^L$, and 
$g \in U_{\mathfrak{I}}^{1, \det =1}$ 
induces the automorphism  
\[
 \overline{\mathfrak{X}^L} \to 
 \overline{\mathfrak{X}^L} ;\ 
 (\bom{z},(\bom{y}_i)_{1 \leq i \leq n-1}) \mapsto
 (\bom{z}+r_L(g) ,(\bom{y}_i)_{1 \leq i \leq n-1}). 
\]
\end{prop}
\begin{proof}
Assume that 
$g=(a_{i,j})_{1 \leq i,j \leq n} \in \iGL_n (K)^0$ 
stabilizes $\mathcal{X}^L$. 
Since $g$ stabilizes $\mathcal{X}^L$, 
the action of $g$ sends a point of 
$\mathcal{M}_{\infty,\bar{\eta}}(\bfC)$ 
to a point of $\mathcal{M}_{\infty,\bar{\eta}}(\bfC)$ 
in the decomposition \eqref{Mdec}. 
Hence we have $\det (g)=1$ by the definition of the action of $g$. 
By Lemma \ref{vgd}, we have $g \in \fI^{\times}$. 

We write 
$a_{i,j}=
\sum_{l=0}^{\infty}a_{i,j}^{(l)}\varpi^l$ with 
$a_{i,j}^{(l)} \in \mu_{q-1} (K) \cup\{0\}$.
By \eqref{gl} and Lemma \ref{appsum}, we have 
\begin{equation}\label{x0}
\begin{gathered}
\begin{split}
 & g^{\ast} \bom{X}\!_1 \equiv 
 a^{(0)}_{1,1} \bom{X}\!_1+a^{(1)}_{n,1} \bom{X}\!_n^{\,q^n} 
 \mod\!_>\, q/(n(q-1)),\\  
 & g^{\ast} \bom{X}\!_i \equiv 
 a^{(0)}_{i,i} \bom{X}\!_i+a_{i-1,i}^{(0)} \bom{X}\!_{i-1} 
 \mod\!_>\, (nq^{i-2}(q-1))^{-1} \quad 
 \textrm{for $2 \leq i \leq n$}.
\end{split}
\end{gathered}
\end{equation}
The former assertion follows from this and \eqref{affc}. 

We prove the latter assertion. 
Assume that $g \in U_{\mathfrak{I}}^{1,\det=1}$. 
Note that $a^{(0)}_{i,i}=1$ for $1 \leq i \leq n$ in this case. 
By \eqref{eq:v(Xi)} and \eqref{x0}, we have 
\begin{equation}\label{eq:gii+1}
 \biggr( \frac{g^{\ast} \bom{X}\!_i}{g^{\ast} \bom{X}\!_{i+1}} 
 \biggr)^{q^{i-1}(q-1)} 
 \equiv 
 \biggr( \frac{\bom{X}\!_i}{\bom{X}\!_{i+1}} 
 \biggr)^{q^{i-1}(q-1)} 
 + a_{i,i+1}^{(0)}
 \biggr( \frac{\bom{X}\!_i}{\bom{X}\!_{i+1}} \biggl)^{q^{i}} 
 \mod\!_{>}\, \frac{1}{n} 
\end{equation}
for $1 \leq i \leq n-1$ and 
\begin{equation}\label{eq:gn1}
 \biggl( \frac{g^{\ast} \bom{X}\!_n^{\,q^n}}{g^{\ast} \bom{X}\!_1} 
 \biggr)^{\frac{q-1}{q}} \equiv 
 \biggl( \frac{\bom{X}\!_n^{\,q^n}}{\bom{X}\!_1} 
 \biggr)^{\frac{q-1}{q}} + a_{n,1}^{(1)}
 \biggr( \frac{\bom{X}\!_n^{\,q^{n}}}{\bom{X}\!_1} \biggl) 
 \mod\!_{>}\, \frac{1}{n}. 
\end{equation}
We set 
\[
 \Delta_g (\bom{X})= 
 \sum_{i=1}^{n-1} a_{i,i+1}^{(0)}
 \biggr( \frac{\bom{X}\!_i}{\bom{X}\!_{i+1}} \biggl)^{q^{i}} 
 +a_{n,1}^{(1)} \biggr( \frac{\bom{X}\!_n^{\,q^{n}}}{\bom{X}\!_1} \biggl). 
\]
Then, by \eqref{Zdef}, \eqref{eq:gii+1} and \eqref{eq:gn1}, we have 
\[
  g^\ast \bom{Z} \equiv \bom{Z}+\Delta_g (\bom{X}) \mod\!_{>}\, \frac{1}{n} . 
\]
Hence, we obtain 
\begin{equation}\label{kou}
 g^\ast \bom{z} \equiv \bom{z}+\eta_L^{-1} \Delta_g (\bom{X}) \mod\!_{>}\, 0 
\end{equation}
by \eqref{yzdef}. 
We have 
\begin{equation*}
 \biggr( \frac{\bom{X}\!_i}{\bom{X}\!_{i+1}} \biggl)^{q^{i}} 
 \equiv 
 \biggr( \frac{\bom{\xi}_{L,i}}{\bom{\xi}_{L,i+1}} \biggl)^{q^{i}} 
 = \bom{\xi}_{L,1}^{q-1} = 
 \bom{\eta}_L \mod\!_>\, \frac{1}{n} 
\end{equation*}
for $1 \leq i \leq n-1$ and 
\begin{equation*}
 \frac{\bom{X}\!_n^{\,q^n}}{\bom{X}\!_1} 
 \equiv 
 \frac{\bom{\xi}_{L,n}^{q^n}}{\bom{\xi}_{L,1}} 
 = \bom{\xi}_{L,1}^{q-1} = 
 \bom{\eta}_L \mod\!_>\, \frac{1}{n} 
\end{equation*}
by \eqref{eq:v(Xi)}. 
Hence, we have 
$\overline{\bom{\eta}_L^{-1} \Delta_g (\bom{X})}=r_L (g)$. 
By this and \eqref{kou}, we obtain 
$g^\ast \bom{z}=\bom{z}+r_L (g)$.
We can easily compute 
the action of $g$ on $\{\bom{y}_i\}_{1 \leq i \leq n-1}$ 
by \eqref{affc} and \eqref{x0}. 
\end{proof}

\subsection{$\mathcal{O}^{\times}_D$-action}
Let $\mathfrak{p}_D$ be the maximal ideal of 
$\mathcal{O}_D$. 
We put 
\[
 U_D ^1 =1+\mathfrak{p}_D \quad \textrm{and} \quad 
 U_D^{1,\Nrd=1}=\{d \in U_D^1 \mid \Nrd_{D/K}(d)=1\}. 
\]
\begin{prop}\label{da}
We set 
$r_L (d)=\overline{\mathrm{Trd}_{D/K}((-\varphi_{D,L})^{-1}(d-1))}$ 
for $d \in U_D^1$.
Then, $U_D^{1,\Nrd=1}$ 
is the stabilizer in $\cO_D^{\times}$ of 
$\mathcal{X}^L$, and $d \in U_D^{1,\Nrd=1}$
induces the automorphism 
\[
 \overline{\mathfrak{X}^L} \to 
 \overline{\mathfrak{X}^L} ;\ 
 (\bom{z},(\bom{y}_i)_{1 \leq i \leq n-1})
 \mapsto 
 (\bom{z}+r_L (d) ,(\bom{y}_i)_{1 \leq i \leq n-1}). 
\]
\end{prop}
\begin{proof}
Let $d \in \mathcal{O}_D^{\times}$. We write $d^{-1}=\sum_{i=0}^{\infty}d_i \varphi_L^i$ with 
$d_i \in \mu_{q^n-1}(K_n) \cup \{0\}$. We set $\kappa(d)=d_1/d_0$. 
By \eqref{divi} and Lemma \ref{appsum}, we have 
\begin{equation}\label{xx2}
 d^\ast \bom{X}\!_i \equiv d_0 \bom{X}\!_i 
 \left(1+ \kappa(d) \bom{X}\!_i^{\,q-1} \right) 
 \mod\!_>\, (nq^{i-2}(q-1))^{-1} \quad 
 \textrm{for $1 \leq i \leq n$}.
\end{equation}
The former assertion follows from \eqref{affc} and \eqref{xx2}. 
By \eqref{eq:v(Xi)} and \eqref{xx2}, we have 
\begin{equation}\label{eq:dii+1}
 \biggr( \frac{d^{\ast} \bom{X}\!_i}{d^{\ast} \bom{X}\!_{i+1}} 
 \biggr)^{q^{i-1}(q-1)} 
 \equiv 
 \biggr( \frac{\bom{X}\!_i}{\bom{X}\!_{i+1}} 
 \biggr)^{q^{i-1}(q-1)} 
 + \kappa(d)^{q^{i-1}} \bom{X}\!_i^{\,q^{i-1}(q-1)} 
 \mod\!_{>}\, \frac{1}{n} 
\end{equation}
for $1 \leq i \leq n-1$ and 
\begin{equation}\label{eq:dn1}
 \biggl( \frac{d^{\ast} \bom{X}\!_n^{\,q^n}}{d^{\ast} \bom{X}\!_1} 
 \biggr)^{\frac{q-1}{q}} \equiv 
 \biggl( \frac{\bom{X}\!_n^{\,q^n}}{\bom{X}\!_1} 
 \biggr)^{\frac{q-1}{q}} + \kappa(d)^{q^{n-1}} \bom{X}\!_n^{\,q^{n-1}(q-1)} 
 \mod\!_{>}\, \frac{1}{n}. 
\end{equation}
We set 
\[
 \Delta_d (\bom{X})=\sum_{i=1}^n \kappa(d)^{q^{i-1}} \bom{X}\!_i^{\,q^{i-1}(q-1)}.
\]
Then we have 
\begin{equation}\label{df0}
 f_0 (d^\ast \bom{X}) \equiv f_0 (\bom{X}) +\Delta_d (\bom{X}) 
 \mod\!_{>}\, \frac{1}{n} 
\end{equation}
by \eqref{f0def}, \eqref{eq:dii+1} and \eqref{eq:dn1}. 
Assume that $d \in U_D^{1,\Nrd=1}$. 
By \eqref{Zdef}, \eqref{yzdef} and \eqref{df0}, 
we obtain
\begin{align}\label{ggdne}
 d^\ast \bom{Z}  \equiv \bom{Z} +\Delta_d (\bom{X}) 
 \mod\!_{>}\, \frac{1}{n}, \quad
 d^\ast \bom{z}  \equiv \bom{z}
 + \bom{\eta}_L^{-1} \Delta_d (\bom{X}) 
 \mod\!_{>}\, 0.
\end{align}
On $\mathcal{X}_L$, we can check that 
\begin{equation}\label{gfc}
 \bom{X}\!_i^{\,q^{i-1}(q-1)}/\bom{\eta}_L \equiv 1
 \mod\!_>\, 0 \quad 
 \textrm{for}\ 1 \leq i \leq n. 
\end{equation}
Note that $(-\varphi_{D,L})^{-1}(d-1) \equiv 
\kappa(d)^{q^{-1}} \mod \varphi \mathcal{O}_D$.  
Therefore, 
by \eqref{ggdne} and \eqref{gfc}, we acquire 
\begin{equation}\label{ggdn}
 d^\ast \bom{z}  \equiv \bom{z}
 + \Trd_{D/K}((-\varphi_{D,L})^{-1}(d-1)) 
 \mod\!_{>}\, 0.
\end{equation}
The action of $d$ on $\{ \bom{y}_i\}_{1 \leq i \leq n-1}$ 
can be easily computed by \eqref{affc}, 
\eqref{Ydef} and \eqref{xx2}.
\end{proof}

\subsection{Action of Weil group}\label{subsec:Weil}
Let $L'$ be the finite separable extension of $K$ 
corresponding to 
$\{ \sigma \in W_K \mid \sigma (L)=L \}$. 
We note that 
$[L:L']=n_q$, where $n_q$ is defined in \eqref{eq:defnq}. 
For $\sigma \in W_K$, 
the point 
$\xi_L^{(1,\varphi^{-n_{\sigma}},\sigma)}$ 
has CM by $L$ 
if and only if $\sigma \in W_{L'}$. 
We define 
$j_L \colon W_{L'} \to L^{\times} \backslash 
 (\iGL_n (K) \times D^{\times} )$ as follows (cf.~\cite[3.1]{BWMax}): 
\begin{quote}
Let $\sigma \in W_{L'}$. 
There exists $(g,d) \in \iGL_n (K) \times D^{\times}$ 
uniquely defined up to left multiplication by $L^{\times}$ 
such that 
$(g,d,\sigma) \in G^0$ and 
$\xi_L^{(g,d,\sigma)}=\xi_L$ by Lemma \ref{CMtrans}. 
We put $j_L (\sigma )=L^{\times} (g,d)$. 
\end{quote}
For $\sigma \in W_{L}$, we put 
$a_{\sigma}=\mathrm{Art}_{L}^{-1}(\sigma) \in L^{\times}$ 
and 
$u_{\sigma}=a_{\sigma} \varphi_{L}^{-n_{\sigma}} 
 \in \mathcal{O}_{L}^{\times}$. 
\begin{lem}\label{jLsig}
For $\sigma \in W_L$, we have 
$j_L (\sigma)=L^{\times} (1,a_{\sigma}^{-1})$. 
\end{lem}
\begin{proof}
This is just \cite[Lemma 3.1.3]{BWMax}. Note that 
our action of $W_K$ is inverse to that in \cite{BWMax}. 
\end{proof}
\begin{prop}\label{LT}
Let $\sigma \in W_{L}$. 
We regard $L$ as a subalgebra of $D$ by $i_{D,\xi_L}$. 
Then, 
$(a_{\sigma}^{-1},\sigma) \in D^{\times} \times W_{L}$ 
stabilizes $\cX^L$, and 
induces 
\[
 \overline{\mathfrak{X}^L} \to \overline{\mathfrak{X}^L} ;\ 
 (\bom{z},(\bom{y}_i)_{1 \leq i \leq n-1})
 \mapsto 
 (\bom{z}^{q^{n_{\sigma}}}, 
 (\ol{u}_{\sigma}^{\frac{q-1}{2}} 
 \bom{y}_i^{q^{n_{\sigma}}})_{1 \leq i \leq n-1}) 
\]
on the coordinates of points. 
\end{prop}
\begin{proof}
Let $P \in \mathcal{X}^L(\mathbf{C})$. 
We have 
\begin{align}
 \bom{Z} (P(a_{\sigma}^{-1},\sigma)) &= 
 f_0 \bigl( \bom{X}(P(a_{\sigma}^{-1},\sigma)) \bigr) - 
 f_0 (\bom{\xi}_L ) \notag \\
 &= f_0 \bigl( \bom{X}(P(a_{\sigma}^{-1},\sigma)) \bigr) - 
 f_0 \bigl( \bom{X}(P(\varphi_L^{-n_{\sigma}},\sigma)) \bigr) + 
 \sigma^{-1} \bigl( f_0 ( \bom{X}(P) ) \bigr) - 
 f_0 (\bom{\xi}_L ) \notag \\ 
 &\equiv \Delta_{u_{\sigma}^{-1}} 
 \bigl( \bom{X}(P(\varphi_L^{-n_{\sigma}},\sigma)) \bigr) + 
 \sigma^{-1} \bigl( \bom{Z}(P) +f_0 (\bom{\xi}_L ) \bigr) - 
 f_0 (\bom{\xi}_L ) \mod\!_{>}\, 1/n \label{ZPsig}
\end{align}
by \eqref{Zdef}, \eqref{df0} and the definition of the action of 
$(\varphi_L^{-n_{\sigma}},\sigma)$ (\cf Subsection \ref{CMpts}). 
Since $\xi_L =\xi_L (\varphi_L^{-n_{\sigma}},\sigma) (u_{\sigma}^{-1},1)$ 
by Lemma \ref{jLsig}, 
we have 
\begin{equation}\label{Delxi}
 f_0 (\bom{\xi}_L ) - f_0 (\sigma^{-1} (\bom{\xi}_L )) 
 \equiv \Delta_{u_{\sigma}^{-1}} (\sigma^{-1} (\bom{\xi}_L )) 
 \equiv 
 \Delta_{u_{\sigma}^{-1}} 
 \bigl( \bom{X}(P(\varphi_L^{-n_{\sigma}},\sigma)) \bigr) 
 \mod\!_{>}\, 1/n
\end{equation}
using \eqref{df0} at the first equality and 
the definition of the action of 
$(\varphi_L^{-n_{\sigma}},\sigma)$ and 
$\bom{X}\!_i(P)/\bom{\xi}_{L,i} \equiv \bom{1} \mod\!_{>}\, 0$ 
for $1 \leq i \leq n$ 
at the second equality. 
Then we have
$\bom{Z} (P(a_{\sigma}^{-1},\sigma)) 
 \equiv \sigma^{-1} ( \bom{Z}(P) ) \mod\!_{>}\, 1/n$ 
by \eqref{ZPsig} and \eqref{Delxi}. 
This implies 
$\bom{z} (P(a_{\sigma}^{-1},\sigma)) 
 \equiv \sigma^{-1} ( \bom{z}(P) ) \mod\!_{>}\, 0$. 
By the same argument using \eqref{xx2}, we have 
$\bom{Y}\!_i(P (a_{\sigma}^{-1},\sigma) ) 
 \equiv \sigma^{-1}(\bom{Y}\!_i(P)) \mod\!_{>}\, 1/(2n)$ 
for $1 \leq i \leq n-1$. 
This implies 
\begin{equation*}
 \bom{y}_i(P (a_{\sigma}^{-1},\sigma) ) \equiv 
 \sigma^{-1}(\bom{\eta}_L^{1/2}) 
 \bom{\eta}_L^{-(1/2)}\sigma^{-1}(\bom{y}_i(P))
 \equiv u_{\sigma}^{(q-1)/2} \bom{y}_i(P)^{q^{n_{\sigma}}} \mod\!_{>}\, 0
\end{equation*}
for $1 \leq i \leq n-1$. 
\end{proof}

\subsection{Action of $\mathbf{g}_L$}
We put  
\begin{equation}\label{g_z}
\mathbf{g}_L=
( \varphi_{M,L},\varphi_{D,L},1 )
 \in G. 
\end{equation}
\begin{prop}\label{frob}
{\rm 1.} 
The action of $\bfg_L$ 
stabilizes $\cX^L$, and 
induces the automorphism of 
$\overline{\mathfrak{X}^L}$ defined by 
\begin{equation}\label{oct}
 \bom{z} \mapsto \bom{z}, \quad 
 \bom{y}_1 \mapsto -\sum_{i=1}^{n-1} \bom{y}_i,\quad 
 \bom{y}_i \mapsto \bom{y}_{i-1} \quad 
 \textrm{for $2 \leq i \leq n-1$}. 
\end{equation}
{\rm 2.} Let $g_L \in 
\mathit{GL}_{n-1}(k)$ 
be the matrix corresponding to the action of 
$\bfg_L$ in \eqref{oct}.
Then, 
the characteristic 
polynomial of $g_L$ equals
$(T^n-1)/(T-1)$.
\end{prop}
\begin{proof}
By \eqref{gl} and \eqref{divi}, 
we have 
\begin{equation*} 
\bfg_L^{\ast} \bom{X}\!_1 = 
  \bom{X}\!_n^{\,q^{n-1}}, \quad 
 \bfg_L^{\ast} \bom{X}\!_i =
 \bom{X}\!_{i-1}^{\,q^{-1}} 
 \quad \textrm{for $2 \leq i \leq n$.}
\end{equation*}
The first assertion follows from this. 
The second assertion is directly checked. 
\end{proof}

\subsection{Stabilizer}\label{subsec:stab}
We put 
\[
 \cS_L = \{ (g,d,\sigma) \in \iGL_n (K) \times D^{\times} 
 \times W_{L'} \mid 
 j_L (\sigma )=L^{\times} (g,d) \}. 
\]
Then, this is the stabilizer of $\xi_L$ 
by the definition of $j_L$. 
\begin{lem}\label{SLact}
The action of $\cS_L$ on $\cM^{(0)}_{\infty,\ol{\eta}}$ stabilizes $\cX^L$, 
and induces the action on the formal model $\fX^L$, 
which gives the action on $\ol{\fX^L}$. 
\end{lem}
\begin{proof}
We take an element of $\cS_L$, 
and write it as $(g,d \varphi_{D,L}^{-n_{\sigma}} ,\sigma )$, 
where $(g,d,1) \in G^0$ and $\sigma \in W_K$. 
Since 
$\xi_L (g,d ,1)=\xi_L (1,\varphi_{D,L}^{n_{\sigma}} ,\sigma^{-1} )$, 
we have 
$(g,d ) \in (\varphi_{M,L},\varphi_{D,L})^l (\fI^{\times} \times \cO_D^{\times} )$ 
by Lemma \ref{vgd}. 

To show the claims, 
we may assume that 
$(g,d ) \in \fI^{\times} \times \cO_D^{\times}$ 
by Proposition \ref{frob}.1. 
Then, the action of 
$(g,d \varphi_{D,L}^{-n_{\sigma}} ,\sigma )$ 
preserves the condition \eqref{affc} 
by the definition of the action in Subsection \ref{gpfm} 
and the fact that 
the action of $(g,d \varphi_{D,L}^{-n_{\sigma}} ,\sigma )$ 
fixes $\xi_L$. 
The both claims follow from this. 
\end{proof}

The conjugation by $\cS_L$ stabilizes 
$U_{\mathfrak{I}}^{1, \det =1} \times 1 \times 1 \subset G$ and 
$1 \times U_D^{1,\Nrd=1} \times 1 \subset G$, 
because 
$U_{\mathfrak{I}}^{1, \det =1}$ and $U_D^{1,\Nrd=1}$ 
are the stabilizers of $\cX^L$ in 
$\iGL_n (K)^0$ and $\cO_D^{\times}$ respectively. 
We put 
\[
 H_L =(U_{\mathfrak{I}}^{1, \det =1} \times U_D^{1,\Nrd=1} \times 1 ) 
 \cdot \cS_L \subset G. 
\]
Then $H_L$ acts on $\ol{\fX^L}$ by Proposition \ref{ga}, 
Proposition \ref{da} and Lemma \ref{SLact}. 

\begin{prop}
The subgroup $H_L \subset G^0$ is the stabilizer of 
$\cX^L$ in $\cM^{(0)}_{\infty,\ol{\eta}}$. 
\end{prop}
\begin{proof}
Assume that 
$(g,d,\sigma) \in G^0$ stabilizes $\cX^L$. 
To show $(g,d,\sigma) \in H_L$, 
we may assume that 
$g \in \iGL_n (K)^0$, $d \in \cO_D^{\times}$ and $\sigma =1$ 
by Lemma \ref{SLact}. 
Then we have 
$(g,d ) \in \fI^{\times} \times \cO_D^{\times}$ 
by Lemma \ref{vgd}. 
We have also $\det (g)^{-1} \Nrd_{D/K}(d) =1$, since 
$(g,d,1)$ stabilizes $\cX^L$. 
Further, we see that 
$(g,d ) \in i_{\xi_L} (\cO_K^{\times}) \cdot 
 (U_{\mathfrak{I}}^{1, \det =1} \times U_D^{1,\Nrd=1})$ 
by \eqref{x0} and \eqref{xx2}. 
Hence, we have the claim. 
\end{proof}

\section{Cohomology of Artin-Schreier variety}\label{CohASv}
\subsection{Artin-Schreier variety}
We put $f =[k:\bF_p]$. 
Let $m$ 
be a positive integer dividing $f$. 
Let $r$ be a positive 
integer such that $r+1$ is prime to $p$. 
For $w=(m,r)$, let $X_{w,0}$ 
be the affine smooth variety of dimension $r$ defined by 
\[
 z^{p^m}-z=
 \sum_{1 \leq i \leq j \leq r} y_i y_j 
 \quad \textrm{in}\ \mathbb{A}_{k}^{r+1}. 
\]
We simply 
write $\nu_r$ for the right hand side of the above equation.
Let $X_w$ denote the base change 
$X_{w,0} \times_{k} k^{\mathrm{ac}}$.
We fix a prime number 
$\ell \neq p$. 
For a variety $Z$ over $k^{\mathrm{ac}}$ and 
an integer $i \geq 0$, 
we simply write $H_{\mathrm{c}}^i(Z)$ for 
$H_{\mathrm{c}}^i(Z,\overline{\mathbb{Q}}_{\ell})$. 
The finite group 
$\mathbb{F}_{p^m}$ acts on $X_{w,0}$ 
by $z \mapsto z+b$
for $b \in \mathbb{F}_{p^m}$. 
Let $\mathrm{Frob}_{q} \in \mathrm{Gal}(k^{\mathrm{ac}}/{k})$ 
be the geometric Frobenius element. 
For a finite abelian group $A$, 
the character group 
$\mathrm{Hom}_{\mathbb{Z}} (A,\overline{\mathbb{Q}}_{\ell} ^{\times})$ 
is denoted by $A^{\vee}$.

\subsection{Odd characteristic case}
We assume that $p$ is odd in this subsection.  
\begin{lem}\label{detnu}
We have 
$\det \nu_r =2^{-r}(r+1)$ in $k^{\times}/(k^{\times})^2$. 
In particular, $\nu_r$ is non-degenerate. 
\end{lem}
\begin{proof}
This follows from a direct computation. 
\end{proof}

For an additive character 
$\psi \colon \mathbb{F}_{p^m} \to 
\overline{\mathbb{Q}}^{\times}_{\ell}$, 
let $\mathscr{L}_{\psi}(s)$ 
denote the Artin-Schreier 
$\overline{\mathbb{Q}}_{\ell}$-sheaf on 
$\mathbb{A}_{k}^1=\Spec k[s]$ associated to $\psi$, 
which is equal to 
$\mathfrak{F} (\psi)$ in the notation of 
\cite[Sommes trig. 1.8 (i)]{DelCoet}. 
We regard 
$\nu_r$ as a 
morphism 
$\nu_r \colon \mathbb{A}_{k}^r \to \mathbb{A}_{k}^1$.
We write $\mathscr{L}_{\psi}(\nu_r)$ for the pull-back
$\nu_r^\ast \mathscr{L}_{\psi}(s)$ to $\mathbb{A}_{k}^r$.
We consider the 
quadratic 
Gauss sum associated to 
$\nu_r$ and a non-trivial character 
$\psi$
\begin{equation}\label{qg}
\mathfrak{g}(\nu_r,\psi)=\sum_{x \in k^r}
\left(\psi \circ \mathrm{Tr}_{k/\mathbb{F}_{p^m}}\right)(\nu_r(x)).
\end{equation}
\begin{prop}\label{cg1}
{\rm 1.} We have a decomposition 
\begin{equation}\label{HXo}
 H_{\mathrm{c}}^r(X_w) \simeq 
 \bigoplus_{\psi \in \mathbb{F}_{p^m}^{\vee} \backslash \{1\}} \psi
\end{equation}
as ${\mathbb{F}_{p^m}}$-representations, 
and 
$H^{i}_{\mathrm{c}}(X_w)=0$ 
for $i \neq r,2r$. \\ 
{\rm 2.} 
Let $\psi \in \mathbb{F}_{p^m}^{\vee} \backslash \{1\}$. 
The element $\mathrm{Frob}_q$ acts on 
the $\psi$-component in \eqref{HXo} as 
multiplication by 
$(-1)^r \mathfrak{g}(\nu_r,\psi)$. \\ 
{\rm 3.} 
Let $g \in \mathit{GL}_r(k)$ be a matrix such that 
$\nu_r(gx)=\nu_r(x)$ for $x \in k^r$.
Then, the matrix $g$ acts on 
$H_{\mathrm{c}}^r(\mathbb{A}_{k^{\rmac}}^r,\mathscr{L}_{\psi}(\nu_r))$ 
as scalar multiplication 
by $\det (g) \in \{\pm1\}$. 
\end{prop}
\begin{proof}
We prove the claim $1$. 
We have an isomorphism 
\[
 H_{\mathrm{c}}^i(X_w) \simeq \bigoplus_{\psi \in \mathbb{F}_{p^m}^{\vee}} 
 H_{\mathrm{c}}^i(\mathbb{A}_{k^{\rmac}}^r,\mathscr{L}_{\psi}(\nu_r)) 
\]
as $\overline{\mathbb{Q}}_{\ell}[{\mathbb{F}_{p^m}}]$-modules.
Therefore, the claim follows from a well-known fact that 
\[
 \dim H_{\mathrm{c}}^i(\mathbb{A}_{k^{\rmac}}^r,\mathscr{L}_{\psi}(\nu_r)) = 
 \begin{cases}
 1 & \textrm{if $i=r$,} \\ 
 0 & \textrm{otherwise} 
 \end{cases}
\]
for any non-trivial character $\psi$ of $\bF_{p^m}$. 

We prove the claim $2$.
Let $F$ be the Frobenius endomorphism of 
$\mathbb{A}_{k}^r$ 
over ${k}$. 
Then, we have 
\[
 \mathrm{Tr} \left(F^{\ast},H_{\mathrm{c}}^r
 (\mathbb{A}_{k^{\rmac}}^r,\mathscr{L}_{\psi}(\nu_r))\right)=(-1)^r 
 \mathfrak{g}(\nu_r,\psi) 
\]
by 
\cite[Sommes trig. Scholie 1.9]{DelCoet}. 
Hence, we acquire the claim $2$ 
by $\mathrm{Frob}_q =F^{\ast}$. 
The claim $3$ 
follows from \cite[Lemma 2.2.3]{DenLoCh}. 
\end{proof}

\subsection{Even characteristic case}
Assume that $p=2$ in this subsection. 
We set $r'=r/2$.
We define $\Psi \colon X_{w,0} \to \mathbb{A}_{k}^r$ by 
\begin{equation}\label{Psidef}
 \Psi \bigr( (z,y_1, \ldots, y_r) \bigl) = 
 \Biggl( \sum_{1 \leq j \leq i} y_j  \Biggr)_{1 \leq i \leq r} . 
\end{equation}
Let $\Phi \colon \mathbb{A}_k^r \to \mathbb{A}_k^r$ 
be the purely inseparable map defined by 
$a_{2i} \mapsto a_{2i}$ and 
$a_{2i-1} \mapsto a_{2i-1}^{2^{m-1}}$ for $1 \leq i \leq r'$. 
We define an affine smooth variety $X'_{w,0}$ over $k$ by 
the cartesian diagram 
\begin{equation}\label{pure}
\begin{split}
\xymatrix{
 X'_{w,0} \ar[r]\ar[d] & X_{w,0} \ar[d]^-{\Psi} \\
 \mathbb{A}_{k}^r \ar[r]^-{\Phi} & \mathbb{A}_{k}^r.}
\end{split}
\end{equation}
Let $X'_w$ denote the base change 
$X'_{w,0} \times_{k} k^{\mathrm{ac}}$. 
We put 
\begin{equation}\label{eq1}
 s=\sum_{i=0}^{m-1}
 (a_{r-1}^{2^i-1}z)^{2^{m-i-1}}
 +\sum_{i=1}^{r'} a_{2i}. 
\end{equation}
\begin{lem}\label{l-f}
The variety $X'_{w,0}$ is isomorphic to
the affine smooth variety defined by 
\[
 s^2 +a_{r-1}^{2^{m-1}} s=
 (a_{r-1}^{2^m -1} -1)z +
 \sum_{i=1}^{r'} a_{2i-1}^{2^m} 
 +\sum_{i=1}^{r'-1} 
 a_{2i} (a_{r-1} +a_{2i+1} +a_{2i-1})^{2^{m-1}} 
 \quad \textrm{in} \quad \mathbb{A}_k^{r+1}. 
\]
\end{lem}
\begin{proof}
By \eqref{Psidef} and \eqref{pure}, the variety 
$X'_{w,0}$ is defined by 
\begin{equation}\label{X'eq}
 z^{2^m}-z=\sum_{i=1}^{r'} a_{2i-1}^{2^m} 
 +\sum_{i=1}^{r'-1} 
 a_{2i} (a_{2i} +a_{2i+1}^{2^{m-1}} +a_{2i-1}^{2^{m-1}}) 
 +a_r (a_r + a_{r-1}^{2^{m-1}} ) 
\end{equation}
in $\mathbb{A}_{k}^{r+1}$. 
Substituting \eqref{eq1} into \eqref{X'eq}, we have the claim. 
\end{proof}

Let 
$S \subset \mathbb{A}_{k^{\rmac}}^{r'}
=\Spec k^{\mathrm{ac}}[a_1, \ldots, a_{r'}]$
be the $0$-dimensional 
closed subscheme 
defined by $a_{r'} \in \mathbb{F}_{2^m}^{\times}$ and 
\[
 a_i=
 \begin{cases}
 a_{r'}  \quad &  \textrm{if $i \equiv r' \pmod{2}$},\\
 0  \quad & \textrm{otherwise}.
\end{cases}
\]
Let $U$ be the complement of $S$
in $\mathbb{A}_{k^{\rmac}}^{r'}$. 
We put 
$Y=\coprod_{S \times \mathbb{F}_2}\mathbb{A}_{k^{\rmac}}^{r'}$.
Then, we have 
a canonical morphism
$Y \to S$.
We define a morphism 
$\pi \colon X'_{w} \to \mathbb{A}_{k^{\rmac}}^{r'}$
by 
$(s,z,a_1,\ldots,a_{r-1}) \mapsto (a_1,a_3,\ldots,a_{r-1})$. 
Let $N_r$ be the cardinality of 
the set which consists of all integers 
$1 \leq l \leq r-1$ such that $r-1 \equiv l
\pmod{4}$. 
We put $s_1 =s/a_{r-1}^{2^{m-1}}$ 
on $\pi^{-1}(S)$. 
Then, we have 
\begin{equation}\label{gv2}
s_1^2-s_1=N_r 
\end{equation}
on $\pi^{-1}(S)$ 
by Lemma \ref{l-f}, 
because $a_{r-1}^{2^m -1} =1$, 
$\sum_{i=1}^{r'} a_{2i-1}^{2^m} =N_r a_{r-1}^{2^m}$ and 
$a_{r-1} +a_{2i+1} +a_{2i-1}=0$ for $1 \leq i \leq r' -1$ on 
$\pi^{-1}(S)$. 
We choose an element 
$\varrho \in k^{\mathrm{ac}}$ such that 
$\varrho^2 -\varrho=N_r$.
\begin{lem}\label{p1}
The projection $\pi \colon X'_{w} \to \mathbb{A}_{k^{\rmac}}^{r'}$
is an affine bundle of relative dimension $r'$ over $U$, and 
the morphism 
\begin{equation}\label{yf}
 \pi^{-1}(S) \to Y;\ 
 x=(s,z,a_1, \ldots,a_{r-1}) \mapsto 
 \bigl( z,(a_{2i})_{1 \leq i \leq r'-1} 
 \bigr)_{(\pi(x),s_1 -\varrho)} 
\end{equation}
is an isomorphism over $S$. 
\end{lem}
\begin{proof}
Outside $\pi^{-1}(S)$, we have 
$a_{r-1}^{2^m -1} -1\neq 0$ or 
$a_{r-1} +a_{2i+1} +a_{2i-1} \neq 0$ for some $1 \leq i \leq r' -1$. 
If $a_{r-1}^{2^m -1} -1\neq 0$, 
we can write $z$ by other parameters 
using the equation in Lemma \ref{l-f}. 
If $a_{r-1} +a_{2i+1} +a_{2i-1} \neq 0$ for some $1 \leq i \leq r' -1$, 
we can write $a_{2i}$ by other parameters similarly. 
Hence, $\pi$ is an affine bundle of relative dimension $r'$ over $U$. 

On the other hand, 
the parameter of $\pi^{-1}(S)$ is given by 
a point of $S$, 
$s_1$ satisfying \eqref{gv2} and 
$\bigl( z,(a_{2i})_{1 \leq i \leq r'-1} \bigr)$ with no relation. 
Hence the second claim follows. 
\end{proof}

The finite group 
$\mathbb{F}_{2^m}$ also acts on  
$X'_{w,0}$ by $z \mapsto z+b$
for $b \in \mathbb{F}_{2^m}$.
Clearly, $\pi^{-1}(S)$ 
is stable under this action.
By the isomorphism 
$\pi^{-1}(S) \simeq Y$ 
in Lemma \ref{p1}, the scheme $Y$ also has the action of 
${\mathbb{F}_{2^m}}$. 
The purely inseparable map $X'_{w,0} \to X_{w,0}$ 
in \eqref{pure} induces an isomorphism 
\begin{equation}\label{pure2}
H_{\mathrm{c}}^i(X_{w}) \xrightarrow{\sim} H_{\mathrm{c}}^i(X'_{w})
\end{equation}
as $\mathbb{F}_{2^m}$-representations for any integer $i$. 
Let $V$ be the complement of $\pi^{-1}(S)$ in $X'_{w}$.
Since 
$\pi|_V \colon V \to U$ is an affine bundle 
of relative dimension $r'$ by Lemma \ref{p1},
we have 
\begin{equation}\label{RpiV}
 R(\pi|_V)_!
 \overline{\mathbb{Q}}_{\ell} \simeq 
 \overline{\mathbb{Q}}_{\ell}(-r')[-2r'] 
\end{equation}
by \cite[VII Proposition 1.1]{SGA5}. 
Hence, 
we have an isomorphism
\begin{equation}\label{co1}
H_{\mathrm{c}}^{r}(X'_{w}) \simeq \ker 
(H_{\mathrm{c}}^{r}(Y) \overset{\delta}{\to} H_{\mathrm{c}}^{r+1}(V))
\end{equation}
as $\mathbb{F}_{2^m}$-representations 
by the localization sequence for 
$\left(X'_{w},\pi^{-1}(S)\right)$. 
For an integer $m_1$ and a positive odd integer $m_2$, 
let $\bigl( \frac{m_1}{m_2} \bigr)$ denote 
the Jacobi symbol. 

\begin{prop}\label{HX2}
{\rm 1.} We have an isomorphism 
\[
 H_{\mathrm{c}}^r(X_{w}) \simeq 
 \bigoplus_{\psi \in \mathbb{F}_{2^m}^{\vee} \backslash \{1\}}\psi 
\] 
as $\mathbb{F}_{2^m}$-representations, and 
$H^i_{\mathrm{c}}(X_{w})=0$ for $i \neq r,2r$. \\ 
{\rm 2.}  
The element 
$\mathrm{Frob}_q$ acts on 
$H_{\mathrm{c}}^r(X_w)$ as 
scalar multiplication by 
$\left(\frac{q}{r+1}\right)q^{r'}$. \\ 
{\rm 3.} 
Let $g$ be a $k$-automorphism of 
$X_{w,0}$ of finite order. 
Assume that $g$ commutes 
with the action of ${\mathbb{F}_{2^m}}$ and 
the order of $g$ is odd. 
Then, $g$ 
acts on $H_{\mathrm{c}}^n(X_w)$
trivially.
\end{prop}
\begin{proof}
We prove the first claim in $1$.
By \eqref{pure2}, 
it suffices to study the right hand side of \eqref{co1}. 
By the isomorphism 
$S \xrightarrow{\sim} \mathbb{F}_{2^m}^{\times}; \  
 (a_1,\ldots,a_{r'}) \mapsto a_{r'}$
 and \eqref{yf}, 
we identify 
$\pi_0(\pi^{-1}(S))$ with 
$\pi_0(Y) = \mathbb{F}_{2^m}^{\times} \times \mathbb{F}_2$. 
Let $b \in 
\mathbb{F}_{2^m}$.
Then, by \eqref{eq1}, the automorphism 
$z \mapsto z+b$
of $X'_w$ induces the automorphism of 
$\pi_0(\pi^{-1}(S)) = \mathbb{F}_{2^m}^{\times} 
\times \mathbb{F}_2$ 
which is defined by 
$(a,x) \mapsto (a,
x+\mathrm{Tr}_{\mathbb{F}_{2^m}/\mathbb{F}_2}(b/a))$
for $(a,x) \in \mathbb{F}_{2^m}^{\times} 
\times \mathbb{F}_2$. 
Hence, 
we obtain an isomorphism 
\[
H_{\mathrm{c}}^{r}(Y) \simeq 
\left(\left(\bigoplus_{\psi \in \mathbb{F}_{2^m}^{\vee} 
\backslash \{1\}}
\psi\right) \oplus 
\overline{\mathbb{Q}}_{\ell}^{\oplus(2^m-1)}\right)(-r')
\]
as $\mathbb{F}_{2^m}$-representations.
We have isomorphisms 
$H_{\mathrm{c}}^{r+1}(V) \simeq H_{\mathrm{c}}^{1}(U)(-r') 
 \simeq \overline{\mathbb{Q}}_{\ell}(-r')^{\oplus(2^m-1)}$
as $\mathbb{F}_{2^m}$-representations by \eqref{RpiV}. 
Then $\delta$ is given by the natural projection 
\begin{equation}\label{pou}
 H_{\mathrm{c}}^r(Y) \simeq 
 \left(\left(\bigoplus_{\psi \in \mathbb{F}_{2^m}^{\vee} \backslash \{1\}}
 \psi \right)\oplus 
 \overline{\mathbb{Q}}_{\ell}^{\oplus(2^m-1)}\right)(-r')
 \overset{\mathrm{pr}}{\longrightarrow}
 \overline{\mathbb{Q}}_{\ell}^{\oplus(2^m-1)}(-r') 
 \simeq 
 H_{\mathrm{c}}^{r+1}(V). 
\end{equation}
Hence, the first claim in $1$ follows.

We prove the second claim in $1$. 
By \eqref{pure2} and \cite[XIV 3.2]{SGA4-3}, 
it suffices to show 
$H_{\mathrm{c}}^i(X'_w)=0$ for 
$r< i< 2r$.
Since $\delta$ 
is surjective, 
we obtain 
$H_{\mathrm{c}}^{r+1}(X'_w) \simeq 0$.
For any 
$r+1 < i < 2r$, 
we have 
$H_{\mathrm{c}}^i(X'_w) \simeq H_{\mathrm{c}}^i(V) \simeq H_{\mathrm{c}}^{i-r}(U)(-r')=0$.
Hence, the second claim in $1$ follows. 

We prove the claim 2. 
Let $P$ be a $k^{\mathrm{ac}}$-valued point of $X'_w$. 
By \eqref{gv2}, we acquire 
$s_1(P)^q-s_1(P)=f N_r$.
Therefore, 
the claim follows from the description \eqref{pou} and 
$(-1)^{N_r}=\bigl( \frac{2}{r+1} \bigr)$. 
 
We prove the claim $3$. In just the same way as the claim $1$, 
we have an isomorphism
\[
 H_{\mathrm{c}}^r(X_w,{\mathbb{Q}}_{\ell}) 
 \simeq \bigoplus_{\psi \in \mathbb{F}_{2^m}^{\vee} \backslash \{1\}}\psi
\] 
as ${\mathbb{F}_{2^m}}$-representations.
Since $g$ commutes with the ${\mathbb{F}_{2^m}}$-action, 
$g$ preserves each $\psi$-part and 
acts on it as multiplication by some scalar 
$\alpha_{\psi,\ell} \in \mathbb{Q}_{\ell}^{\times}$.
We have 
$\alpha_{\psi,\ell} \in \mu_{\ell-1}(\mathbb{Q}_{\ell})$, because 
$g$ has a finite order. 
we prove that $\alpha_{\psi,\ell}=1$ 
for any $\psi$ and $\ell \neq 2$. 
Let $n_g$ be the order of $g$.
By the assumption 
that $n_g$ is odd, 
we choose an odd prime number 
$\ell'$ such that 
$(n_g,\ell'-1)=1$.
Because 
$\alpha_{\psi,\ell'} \in \mu_{\ell'-1}(\mathbb{Q}_{\ell'})$ and 
$(n_g,\ell'-1)=1$,
we obtain $\alpha_{\psi,\ell'}=1$.
We prove $\alpha_{\psi,\ell}=1$ 
for any odd prime $\ell$.
By \cite[3.5(c)]{IllMtr}, 
we acquire 
\[
\mathrm{Tr} (g_{\ast}, H_{\mathrm{c}}^r(X_w,\mathbb{Q}_{\ell}))
=\mathrm{Tr} (g_{\ast}, H_{\mathrm{c}}^r(X_w,\mathbb{Q}_{\ell'}))=2^m-1.
\]
Hence, we have 
$\sum_{\psi \in \mathbb{F}_{2^m}^{\vee} \backslash \{1\}}\alpha_{\psi,\ell}=2^m-1$.
Since any 
$\alpha_{\psi,\ell}$ is a root of unity, 
we obtain 
$\alpha_{\psi,\ell}=1$ 
for any $\psi$
and $\ell \neq 2$.
\end{proof}

\begin{rem}
Later, we apply Proposition \ref{HX2}.3 to 
$\bfg_L$ in \eqref{g_z}, 
whose action does not preserve the fibration 
in Lemma \ref{p1}. 
\end{rem}

\section{Realization of LLC and LJLC}\label{RealLLCLJLC}
\subsection{Explicit descriptions of the correspondences}\label{ExpCorr}
\begin{defn}
We say that a smooth irreducible supercuspidal representation 
of $\iGL_n (K)$ is simple supercuspidal 
if its exponential Swan conductor is one. 
We apply the same definition to 
a smooth irreducible representation of $D^{\times}$. 
\end{defn}

We fix a uniformizer $\varpi$ of $K$. 
Let $\zeta \in \mu_{q-1} (K)$. 
We take $\varphi_{\zeta} \in \bfC$ such that 
$\varphi_{\zeta}^n =\zeta \varpi$. 
We put $L_{\zeta}=K(\varphi_{\zeta})$. 
We fix  
a non-trivial character 
\[
 \psi \colon k \to 
 \overline{\mathbb{Q}}_{\ell}^{\times}. 
\]
We take an additive character 
\[
 \psi_K \colon K \to 
 \overline{\mathbb{Q}}_{\ell}^{\times} 
\]
such that 
$\psi_K (x) =\psi(\bar{x})$ 
for $x \in \mathcal{O}_K$. 
We put 
\[
 \varphi_{M,\zeta} = 
 \begin{pmatrix}
 \bm{0} & I_{n-1} \\
 \zeta \varpi & \bm{0} \\
 \end{pmatrix}
 \in M_n (K). 
\]

We regard $L_{\zeta}$ as a sub-$K$-algebra of $M_n(K)$ 
by $\varphi_{\zeta} \mapsto \varphi_{M,\zeta}$. 
Let 
$\chi \in (k^{\times})^{\vee}$ and 
$c \in \overline{\mathbb{Q}}_{\ell}^{\times}$. 
We define a character 
$\Lambda_{\zeta,\chi,c} \colon L_{\zeta}^{\times} U_{\mathfrak{I}}^1 
 \to \overline{\mathbb{Q}}_{\ell}^{\times}$ by 
\begin{align*}
 &\Lambda_{\zeta,\chi,c} (\varphi_{M,\zeta})=c, \quad 
 \Lambda_{\zeta,\chi,c} (x) =\chi (\bar{x}) \quad  \textrm{ for 
 $x \in \mu_{q-1} (K)$}, \\ 
 &\Lambda_{\zeta,\chi,c} (x)=
 (\psi_K \circ 
 \mathrm{tr} )( \varphi_{M,\zeta}^{-1}(x-1)) \quad 
 \textrm{for $x \in U_{\mathfrak{I}}^1$}. 
\end{align*}
We put 
\[
 \pi_{\zeta,\chi,c} = 
 \mathrm{c\mathchar`-Ind}_{L_{\zeta}^{\times} U_{\fI}^1}^{\iGL_n(K)} 
 \Lambda_{\zeta,\chi,c}. 
\]
Then, $\pi_{\zeta,\chi,c}$ 
is a simple supercuspidal representation of $\mathit{GL}_n (K)$, 
and 
every simple supercuspidal representation of $\mathit{GL}_n (K)$ 
is isomorphic to 
$\pi_{\zeta,\chi,c}$ for a uniquely determined 
$(\zeta,\chi,c) \in \mu_{q-1} (K) \times (k^{\times})^{\vee} 
 \times \overline{\mathbb{Q}}_{\ell}^{\times}$ 
by \cite[2.1 and 2.2]{BHLepi}. 

We take an embedding $L_{\zeta} \hookrightarrow D$ as $K$-algebras. 
We write $\varphi_{D,\zeta}$ for the image of $\varphi_{\zeta}$ 
under the embedding 
$L_{\zeta} \hookrightarrow D$. 
We define a character 
$\theta_{\zeta,\chi,c} \colon L_{\zeta}^{\times} U_D ^1 
 \to \overline{\mathbb{Q}}_{\ell}^{\times}$ by
\begin{align*}
 &\theta_{\zeta,\chi,c} (\varphi_{D,\zeta})=(-1)^{n-1}c, \quad 
 \theta_{\zeta,\chi,c} (x) =\chi (\bar{x}) \quad \textrm{for 
 $x \in \mu_{q-1} (K)$}, \\  
 &\theta_{\zeta,\chi,c} (d) =
 \left(\psi_K \circ 
 \mathrm{Trd}_{D /K}\right)( \varphi_{D,\zeta}^{-1}(d-1)) \quad 
 \textrm{for $d \in U_D ^1$}. 
\end{align*}
We put 
\[
 \rho_{\zeta, \chi,c} = 
 \mathrm{Ind}_{L_{\zeta}^{\times} U_D ^1}^{D^{\times}} 
 \theta_{\zeta, \chi,c}. 
\]
Then $\rho_{\zeta, \chi,c}$ is a simple supercuspidal representation 
by \cite[Proposition 2.6]{ABPSdep}. 
The isomorphism class of $\rho_{\zeta, \chi,c}$ does not depend 
on the choice of the embedding $L_{\zeta} \hookrightarrow D$. 
Every simple supercuspidal representation of $D^{\times}$ 
is isomorphic to 
$\rho_{\zeta,\chi,c}$ for a uniquely determined 
$(\zeta,\chi,c) \in \mu_{q-1} (K) \times (k^{\times})^{\vee} 
 \times \overline{\mathbb{Q}}_{\ell}^{\times}$ 
by \cite[Proposition 1.3]{ITsimpJL}. 

Let $E$ be a finite separable extension of $K$. 
We put 
$\psi_E=\psi_K \circ \mathrm{Tr}_{E/K}$. 
Let 
\begin{equation}\label{R}
 R_{E/K}=\mathrm{Ind}_{E/K}1_E, \quad 
 \delta_{E/K}=\det R_{E/K}, 
\end{equation}
where $1_E$ denote the trivial representation of $W_E$. 
Note that $\delta_{E/K}^2 =1$ (cf.~\cite[29.2]{BHLLCGL2}). 
By the local class field theory, we view $\delta_{E/K}$
as a character of $K^{\times}$. 
For a semi-simple smooth 
representation $\rho$ of $W_E$,
let $\epsilon(\rho,s,\psi_E)$ denote the 
Langlands-Deligne local constant of 
$\rho$ (cf.\ \cite[Theorem 29.4]{BHLLCGL2}).
We set 
\[
 \lambda_{E/K}(\psi_K)
 =\frac{\epsilon(R_{E/K},s,\psi_K)}{\epsilon(1_E,s,\psi_E)}, 
\]
which we call the Langlands constant 
of $E$ over $K$ (cf.\ \cite[30.4]{BHLLCGL2}). 
We write $k_E$ 
for the residue field of $E$. 
For $x \in k_E^{\times}$, let 
$\bigl( \frac{x}{k_E} \bigr)$ 
denote the quadratic residue symbol 
of $k_E$. 

\begin{lem}\label{deltame}
Let $E$ be a totally tamely ramified extension of $K$ 
of odd degree $e$. 
Then $\delta_{E/K}$ is the unramified character 
satisfying 
$\delta_{E/K} (\varpi)=\bigl( \frac{q}{e} \bigr)$. 
\end{lem}
\begin{proof}
First assume that the characteristic of $K$ is positive. 
Take a non-archimedean local field $K_{(0)}$ of characteristic zero 
such that the residue field of $K_{(0)}$ is isomorphic to $k$. 
Then we have a correspondence between 
totally tamely ramified extensions of $K$ and $K_{(0)}$ 
by \cite[3.4]{Delp0}. 
Let $E_{(0)}$ be the totally tamely ramified extensions of $K_{(0)}$ 
corresponding to $E$. 
Then $\delta_{E_{(0)}/K_{(0)}}$ corresponds to 
$\delta_{E/K}$ under the isomorphism 
\cite[(3.5.1)]{Delp0}. 
Hence, the claim is reduced to the case where 
the characteristic of $K$ is zero. 
In this case, it is proved in \cite[(10.1.6)]{BFGdiv}. 
\end{proof}

\begin{lem}\label{delres}
Let $E$ be a totally tamely ramified extension of $K$ 
of degree $e$. 
Then we have $\delta_{E/K}(x) =\bigl( \frac{\ol{x}}{k} \bigr)^{e-1}$ 
for $x \in \cO_K^{\times}$. 
\end{lem}
\begin{proof}
If $e$ is odd, the claim follows from Lemma \ref{deltame}. 
Assume that $e$ is even. 
Take the subextension $E'$ of $E$ over $K$ 
satisfying $[E:E']=2$. 
Then we have 
$\delta_{E/K}=\delta_{E/E'}|_{K^{\times}}$ 
by \cite[29.2 Proposition]{BHLLCGL2}. 
The claim follows from this, because 
$\delta_{E/E'}$ is the non-trivial character 
$E'^{\times} \to E'^{\times}/ \Nr_{E/E'}(E^{\times}) 
 \simeq \{ \pm 1 \}$. 
\end{proof}

We define a character $\mu_{\zeta}$ of 
$L_{\zeta}^{\times}$
by 
\begin{equation}\label{mud}
 \mu_{\zeta}|_{U_{L_{\zeta}}^1}=1, \hspace{1.0em} 
 \mu_{\zeta}|_{K^{\times}}=
 \delta_{L_{\zeta}/K}, \hspace{1.0em}
 \mu_{\zeta}(\varphi_{\zeta})= 
 \lambda_{L_{\zeta}/K}(\psi_K).
\end{equation}
We set 
$\xi_{\zeta,\chi,c}=\Lambda_{\zeta,\chi,c}|_{L_{\zeta}^{\times}}$. 
We view a character of $L_{\zeta}^{\times}$ 
as a character of $W_{L_{\zeta}}$ by the local class field theory. 
We put 
\[
 \tau_{\zeta,\chi,c}=
 \mathrm{Ind}_{L_{\zeta}/K}(\mu_{\zeta}^{-1}\xi_{\zeta,\chi,c}). 
\] 
Then $\tau_{\zeta,\chi,c}$ is an irreducible representation 
of exponential Swan conductor one 
by \cite[1.3 Lemma]{BHLepi}. 

Let $\mathrm{LL}$ and $\mathrm{JL}$ denote the local 
Langlands correspondence and 
the local Jacquet-Langlands correspondence 
for $\mathit{GL}_n(K)$ respectively. 

\begin{prop}\label{expJL}
We have 
$\mathrm{LL}(\pi_{\zeta,\chi,c})=\tau_{\zeta,\chi,c}$ and 
$\mathrm{JL} (\rho_{\zeta, \chi,c}) =\pi_{\zeta,\chi,c}$. 
\end{prop}
\begin{proof}
This follows from \cite[Theorem 2.1]{BHestII} 
and \cite[Theorem 5.3]{BHestJL}. 
\end{proof}

\begin{rem}
Explicit descriptions of 
the local Langlands correspondence and 
the local Jacquet-Langlands correspondence 
for simple supercuspidal representations of 
$\mathit{GL}_n(K)$ are obtained in 
\cite{ITlgsw1} and \cite{ITsimpJL} 
also in non-essentially tame cases. 
\end{rem}

\begin{defn}
We say that a smooth irreducible supercuspidal representation 
of $\iGL_n (K)$ is essentially simple supercuspidal 
if it is a character twist of a simple supercuspidal representation. 
We apply the same definition to 
a smooth irreducible representation of $D^{\times}$. 
\end{defn}
Let $\omega \colon K^{\times} \to \ol{\bQ}_{\ell}^{\times}$ 
be a smooth character. 
We put 
\[
 \pi_{\zeta,\chi,c,\omega} =\pi_{\zeta,\chi,c} \otimes (\omega \circ \det), 
 \quad 
 \rho_{\zeta,\chi,c,\omega} =\rho_{\zeta,\chi,c} \otimes (\omega \circ \Nrd_{D/K}), 
 \quad 
 \tau_{\zeta,\chi,c,\omega} =\tau_{\zeta,\chi,c} \otimes 
 (\omega \circ \Art_K^{-1}), 
\]
and 
\begin{align*}
 \Lambda_{\zeta,\chi,c,\omega} &=\Lambda_{\zeta,\chi,c} 
 \otimes (\omega \circ \det |_{L_{\zeta}^{\times} U_{\fI}^1}), 
 \quad
 \theta_{\zeta,\chi,c,\omega} =\theta_{\zeta,\chi,c} 
 \otimes (\omega \circ \Nrd_{D/K}|_{L_{\zeta}^{\times} U_D^1}), 
 \\
 \xi_{\zeta,\chi,c,\omega} &=\xi_{\zeta,\chi,c} \otimes 
 (\omega \circ \Nr_{L/K} \circ \Art_L^{-1}). 
\end{align*}
Then we have 
\[
 \pi_{\zeta, \chi,c,\omega} = 
 \mathrm{c\mathchar`-Ind}_{L_{\zeta}^{\times} U_{\fI}^1}^{\iGL_n(K)} 
 \Lambda_{\zeta, \chi,c,\omega}, \quad 
 \rho_{\zeta, \chi,c,\omega} = 
 \mathrm{Ind}_{L_{\zeta}^{\times} U_D ^1}^{D^{\times}} 
 \theta_{\zeta, \chi,c,\omega}, \quad 
 \tau_{\zeta,\chi,c,\omega}=
 \mathrm{Ind}_{L_{\zeta}/K}(\mu_{\zeta}^{-1}\xi_{\zeta,\chi,c,\omega}). 
\]

\begin{cor}\label{expJLes}
We have 
$\mathrm{LL}(\pi_{\zeta,\chi,c,\omega})=\tau_{\zeta,\chi,c,\omega}$ and 
$\mathrm{JL} (\rho_{\zeta, \chi,c,\omega}) =\pi_{\zeta,\chi,c,\omega}$. 
\end{cor}
\begin{proof}
This follows from Proposition \ref{expJL}, 
because $\mathrm{LL}$ and $\mathrm{JL}$ are compatible with character twists. 
\end{proof}

\subsection{Realization}\label{subsecReal}
Let $L$ be a 
totally ramified extension of $K$ of degree $n$ in $\bfC$. 
We choose $\varpi_L$ in Subsection \ref{ConstAff} so that 
$\varpi_L \varpi^{-1} \in \mu_{q-1}(K)$. 
We put $\zeta_L =\varpi_L \varpi^{-1}$. 
By abuse of notation, 
we simply write $L$ for 
the images of the embeddings 
$i_{M,\xi_L}$ and $i_{D,\xi_L}$. 
We choose an isomorphism 
$\iota \colon \overline{\mathbb{Q}}_{\ell} \simeq \mathbb{C}$. 
Let $q^{1/2} \in  \overline{\mathbb{Q}}_{\ell}$ be the 
square root of $q$ such that 
$\iota(q^{1/2})>0$. 
We put 
\[
 \Pi_{\fX^L} =H_{\mathrm{c}}^{n-1}(\overline{\fX^L})((n-1)/2) 
 \quad \textrm{and} \quad 
 \Pi_L = \mathrm{c\mathchar`-Ind}_{H_L}^G 
 \Pi_{\fX^L}, 
\]
where $((n-1)/2)$ means the twist by 
the unramified character that sends 
a lift of $\mathrm{Frob}_q$ 
to $q^{(1-n)/2}$. 
In the following, we will show that 
this representation realizes the LLC and the LJLC. 

\begin{rem}\label{onlyL}
A different choice of  $\varpi_L$ gives a different $\cX^L$. 
However, it's a translation of the original one under the action of 
$(\iGL_n (K)\times D^{\times})^0$ by Lemma \ref{CMtrans} and 
the construction. Hence, the $G$-representation $\Pi_L$ 
depends only on $L$ and not on $\varpi_L$. 
\end{rem}

For simplicity, we write
$G_1$ and $G_2$ for $\mathit{GL}_n(K)$ and $D^{\times} \times W_K$
respectively, and consider them as 
subgroups of $G$. 
Let $\ol{H_L}$ be the 
image of $H_L$ under the natural projection 
$G \to G_2$. 
We put $H=U_{\fI}^{1,\det =1}$. 
Note that $H=H_L \cap G_1$. 

For $a \in \mu_{q-1} (K)$, 
we define a character 
$\Lambda_{\zeta_L}^a \colon U_{\mathfrak{I}}^1 \to 
 \overline{\mathbb{Q}}_{\ell}^{\times}$ 
by 
$x \mapsto (\psi_K \circ 
 \mathrm{tr})( (a \varphi_{M,L})^{-1}(x-1))$. 

\begin{lem}\label{lf} 
Let $a \in \mu_{q-1} (K)$, 
and $\pi$ be a smooth irreducible supercuspidal representation of 
$\iGL_n (K)$. 
Then, we have 
$\Hom_H ( \Lambda_{\zeta_L}^a,\pi )=0$ if 
$\pi$ is not essentially simple supercuspidal. 
Further, we have 
\[
 \dim\mathrm{Hom}_{H}
  ( \Lambda_{\zeta_L}^a,\pi_{\zeta,\chi,c,\omega} )=
 \begin{cases}
  1\quad  & \textrm{if $a^n \zeta_L=\zeta$},  \\
  0\quad  &  \mathrm{otherwise}.
 \end{cases}
\]
\end{lem}
\begin{proof}
Let $\omega_{\pi}$ denote the central character of $\pi$. 
Twisting $\pi$ by a character, we may assume that 
$\omega_{\pi}$ is tame, 
since $U_K^1 \to U_K^1 \colon x \mapsto x^n$ is an isomorphism. 
Let 
$\Lambda_{\zeta_L,\omega_{\pi}}^a \colon 
 K^{\times}U^1_{\mathfrak{I}} \to \ol{\bQ}_{\ell}^{\times}$ 
be the character such that 
\[ 
 \Lambda_{\zeta_L,\omega_{\pi}}^a (x)=\omega_{\pi} (x) \quad 
 \textrm{for $x \in K^{\times}$}, \quad 
 \Lambda_{\zeta_L,\omega_{\pi}}^a|_{U^1_\fI}=\Lambda_{\zeta_L}^a. 
\]
Then we have 
\begin{equation}\label{LamaHom}
 \mathrm{Hom}_{H} ( \Lambda_{\zeta_L}^a,\pi ) 
 \simeq 
 \mathrm{Hom}_{K^{\times}U^1_{\fI}} 
 ( \Lambda_{\zeta_L,\omega_{\pi}}^a, \pi ) 
 \simeq 
 \mathrm{Hom}_{L^{\times} U_{\mathfrak{I}}^1}
 \bigl( \mathrm{Ind}_{K^{\times} U_{\mathfrak{I}}^1}
 ^{L^{\times} U_{\mathfrak{I}}^1} 
 \Lambda_{\zeta_L,\omega_{\pi}}^a , 
 \pi \bigr)
\end{equation}
by $K^{\times}H=K^{\times }U_{\mathfrak{I}}^1$ and 
the Frobenius reciprocity. 
We take $\chi' \in (k^{\times})^{\vee}$ so that 
$\chi'(\bar{x})=\omega_{\pi} (x)$ for $x \in \mu_{q-1} (K)$. 
For $c' \in \ol{\bQ}_{\ell}^{\times}$, 
we define the character $\Lambda^a_{\zeta_L,\chi',c'} \colon L^{\times}U_{\mathfrak{I}}^1
\to \overline{\mathbb{Q}}_{\ell}^{\times}$
by 
\[
 \Lambda^a_{\zeta_L,\chi',c'} (\varphi_{M,L})=c', \quad
 \Lambda^a_{\zeta_L,\chi',c'} (x)=\chi'(\bar{x}) \quad
 \textrm{ for $x \in \mu_{q-1}(K)$}, \quad 
 \Lambda^a_{\zeta_L,\chi',c'} |_{U_{\mathfrak{I}}^1}=\Lambda_{\zeta_L}^a, 
\]
and set $\pi^a_{\zeta_L,\chi',c'}= 
\mathrm{c\mathchar`-Ind}_{L^{\times} U_{\mathfrak{I}}^1}^{G_1} 
 \Lambda^a_{\zeta_L,\chi',c'}$.
Then, we have 
\begin{equation}\label{KULU}
 \mathrm{Ind}_{K^{\times} U_{\mathfrak{I}}^1}^{L^{\times} 
 U_{\mathfrak{I}}^1} 
 \Lambda_{\zeta_L,\omega_{\pi}}^a \simeq 
 \bigoplus_{c'^n=\omega_{\pi} (\varpi_L )}\Lambda^a_{\zeta_L,\chi',c'}. 
\end{equation}
By \eqref{LamaHom}, \eqref{KULU} and the Frobenius reciprocity, 
we have 
\[
 \mathrm{Hom}_{H} ( \Lambda_{\zeta_L}^a,\pi ) \simeq 
 \bigoplus_{c'^n=\omega_{\pi} (\varpi_L )} 
 \mathrm{Hom}_{G_1} ( \pi^a_{\zeta_L,\chi',c'},\pi ) 
 \simeq 
 \bigoplus_{c'^n=\omega_{\pi} (\varpi_L )} 
 \mathrm{Hom}_{G_1} ( \pi_{a^n \zeta_L,\chi',c'},\pi ) 
\]
because 
$\pi^a_{\zeta_L,\chi',c'} \simeq \pi_{a^n \zeta_L,\chi',c'}$ 
by \cite[2.2 Lemma and Proposition]{BHLepi}. 
Hence, we have the first claim. 
To show the second claim, we may assume that $\omega =1$. 
Then, the claim follows from the above arguments. 
\end{proof}
\begin{prop}\label{ky}
Assume that $p$ is odd. 
Then, we have 
\begin{equation}\label{ky2}
 \lambda_{L/K}(\psi_K) 
 =
 \mathfrak{g}(\nu_{n-1},\psi ) q^{-\frac{n-1}{2}}, 
\end{equation}
where we use notations introduced in \eqref{qg} and 
Subsection \ref{ExpCorr}. 
\end{prop}
\noindent
We give a proof of Proposition \ref{ky} 
in the next subsection. 
\begin{lem}\label{pro} 
{\rm 1.} 
Let $\pi$ be a smooth irreducible supercuspidal representation of 
$\iGL_n (K)$. 
Then, we have 
$\Hom_H ( \Pi_{\fX^L} ,\pi )=0$ if 
$\pi$ is not essentially simple supercuspidal. 
Further, we have 
\begin{equation}\label{n_qq}
 \dim\mathrm{Hom}_{H}
 ( \Pi_{\fX^L},\pi_{\zeta,\chi,c,\omega} )=
 \begin{cases}
 n_q \quad  &\textrm{if $\zeta_L \zeta^{-1} \in \mu_{(q-1)/n_q} (K)$},\\
 0 \quad &  \textrm{otherwise}.
 \end{cases} 
\end{equation}
{\rm 2.} 
We have $L^{\times}U_D^1 \times 
W_{L} \subset \ol{H_L}$ 
and an injective homomorphism 
\[
 \theta_{\zeta_L,\chi,c,\omega} \otimes 
 \mu_{\zeta_L}^{-1}\xi_{\zeta_L,\chi,c,\omega} \hookrightarrow
 \mathrm{Hom}_H 
 ( \Pi_{\fX^L},\pi_{\zeta_L,\chi,c,\omega} )
\]
as $L^{\times}U_D^1 \times W_{L}$-representations.  
\end{lem}
\begin{proof}
By Proposition \ref{ga}, 
Proposition \ref{cg1}.1 
and Proposition \ref{HX2}.1, 
we have 
\begin{equation}\label{fn}
 \Pi_{\fX^L} \simeq 
 \bigoplus_{a \in \mu_{q-1} (K)} \Lambda_{\zeta_L}^a
\end{equation}
as $H$-representations. 
Then we have the claim 1 and the isomorphism 
\begin{equation}\label{as}
 \mathrm{Hom}_{H}
 (\Pi_{\fX^L},\pi_{\zeta,\chi,c,\omega} ) 
 \simeq 
 \bigoplus_{a \in \mu_{q-1} (K),\, a^n \zeta_L =\zeta} 
 \mathrm{Hom}_{H} 
 ( \Lambda_{\zeta_L}^a , \pi_{\zeta,\chi,c,\omega} ) 
\end{equation}
by Lemma \ref{lf} and \eqref{fn}. 

We prove the claim 2. 
We show that the subspace 
\begin{equation}\label{sub1}
 \Hom_{H} (\Lambda_{\zeta_L}^1 ,\pi_{\zeta_L,\chi,c,\omega} ) 
 \subset \Hom_{H} (\Pi_{\fX^L},\pi_{\zeta_L,\chi,c,\omega} ) 
\end{equation}
under the decomposition \eqref{as} is isomorphic to 
$\theta_{\zeta_L,\chi,c,\omega} \otimes 
 \mu_{\zeta_L}^{-1}\xi_{\zeta_L,\chi,c,\omega}$ 
by checking the action of generators of 
$L^{\times}U_D^1 \times W_{L}$. 
Since the subspace \eqref{sub1} is one-dimensional, 
we have 
\begin{equation}\label{piLam}
 \Hom_{H} (\Lambda_{\zeta_L}^1 , \Lambda_{\zeta_L,\chi,c,\omega} ) 
 \simeq
 \Hom_{H} (\Lambda_{\zeta_L}^1,\pi_{\zeta_L,\chi,c,\omega} ). 
\end{equation}

The element $(\varphi_{D,L},1) \in L^{\times}U_D^1 \times W_{L}$ 
acts on 
$\theta_{\zeta_L,\chi,c,\omega} \otimes 
 \mu_{\zeta_L}^{-1}\xi_{\zeta_L,\chi,c,\omega}$ 
as multiplication by 
$(-1)^{n-1} \omega ((-1)^{n-1} \varpi_L )c$. 
The element $\mathbf{g}_L \in H_L$ 
is a lift of $(\varphi_{D,L},1)$ under 
$H_L \to \ol{H_L}$. 
By Proposition \ref{frob}.2, 
Proposition \ref{cg1}.3, 
Proposition \ref{HX2}.3 and \eqref{piLam}, 
the element $\mathbf{g}_L$ acts on 
the subspace \eqref{sub1} 
as scalar multiplication by $(-1)^{n-1} \omega ((-1)^{n-1} \varpi_L ) c$. 

Let $zd \in \mathcal{O}_K^{\times}U_D^1$ with $z \in \mu_{q-1} (K)$
and $d \in U_D^1$. 
The element 
$(zd,1) \in L^{\times}U_D^1 \times W_{L}$ 
acts on $\theta_{\zeta_L,\chi,c} 
 \otimes \mu_{\zeta_L}^{-1}\xi_{\zeta_L,\chi,c}$
as scalar multiplication by
$\chi(\bar{z}) \theta_{\zeta_L,\chi,c}(d) \omega (\Nrd_{D/K}(zd))$. 
Since $n$ is prime to $p$, 
we can choose an element 
$u \in U_K^1$ 
such that $u^n =\mathrm{Nrd}_{D/K}(d)$. 
Then, we have $(zu,zd,1) \in H_L$. 
By Remark \ref{Ktri}, Proposition \ref{da}, 
Proposition \ref{cg1}.1 and Proposition \ref{HX2}.1, 
the element $(zu,zd,1)$ 
acts on the subspace \eqref{sub1} 
as multiplication by 
$\chi(\bar{z}) \theta_{\zeta_L,\chi,c} (d) \omega ((zu)^n)$. 

Let $\sigma \in W_L$ such that $n_{\sigma} =1$. 
The element $(a_{\sigma}^{-1},\sigma) \in L^{\times}U_D^1 \times W_{L}$ 
acts on 
$\theta_{\zeta_L,\chi,c,\omega} \otimes 
 \mu_{\zeta_L}^{-1} \xi_{\zeta_L,\chi,c,\omega}$ 
as scalar multiplication by 
$(-1)^{n-1} \mu_{\zeta_L}(a_{\sigma})^{-1}$. 
We note that the 
$\Lambda_{\zeta_L}^1$-component in \eqref{fn} 
corresponds to the 
$\psi$-component in \eqref{HXo}. 
By Proposition \ref{LT}, 
Proposition \ref{cg1} and Proposition \ref{HX2}, 
the element 
$(1,a_{\sigma}^{-1},\sigma) \in H_L$ 
acts on the subspace \eqref{sub1} 
as scalar multiplication by 
\begin{equation}\label{sigsc}
 \begin{cases}
 (-1)^{n-1}
 \bigl( \frac{\ol{u}_{\sigma}}{k} \bigr)^{n-1}
 \mathfrak{g}(\nu_{n-1},\psi )^{-1}q^{\frac{n-1}{2}}
 \quad & \textrm{if $p\neq 2$},\\ 
 (-1)^{n-1}
 \bigl( \frac{q}{n} \bigr) \quad  & \textrm{if $p=2$.}
 \end{cases}
\end{equation}
We have 
$\mu_{\zeta_L}(\varphi_{\zeta_L})=\lambda_{{L}/K}(\psi_K)$
and 
$\mu_{\zeta_L}(x) =\bigl(\frac{\bar{x}}{k}\bigr)^{n-1}$ 
for $x \in \mathcal{O}^{\times}_K$
by Lemma \ref{delres} and \eqref{mud}. 
Hence, \eqref{sigsc} coincides with 
$(-1)^{n-1} \mu_{\zeta_L}(a_{\sigma})^{-1}$ by 
Proposition \ref{ky} and 
\cite[Theorem 2.1(1)]{BHestII}. 
\end{proof}

\begin{lem}\label{Hindex}
We have $[G_2:\overline{H_L}]=n(q^n-1)/\bigl( n_q(q-1) \bigr)$. 
\end{lem}
\begin{proof}
Let $\ol{H}_{W_L}$ be the inverse image of 
$W_{L}$ under the natural surjection 
$\ol{H_L} \to W_{L'}$, 
where $L'$ is defined in Subsection \ref{subsec:Weil}. 
Then we have 
\[
 [G_2:\overline{H_L}] 
 =[D^{\times} \times W_{L} :\overline{H}_{W_L}] [W_K : W_{L'}] 
 =\frac{n [\cO_D^{\times} :\cO_{L}^{\times} U_D^{1,\Nrd =1}]}{n_q} , 
\]
since we have 
$\ol{H}_{W_L} =L^{\times} U_D^{1,\Nrd =1} \times W_{L}$ 
by Lemma \ref{jLsig}. 
Hence, it suffices to show that 
\[
 [\cO_{D}^{\times}  : 
 \cO_{L}^{\times} U_D^{1,\Nrd =1}]= \frac{q^n -1}{q-1}. 
\] 
This follows from 
$\cO_{L}^{\times} U_D^{1,\Nrd =1} =\cO_{L}^{\times} U_D$. 
\end{proof}

\begin{lem}\label{map} 
Let $\pi$ be a smooth irreducible supercuspidal representation of 
$\iGL_n (K)$. 
Then, we have 
$\mathrm{Hom}_{G_1} ( \Pi_L ,\pi )=0$ if 
$\pi$ is not essentially simple supercuspidal. 
Further, we have 
\[
 \mathrm{Hom}_{G_1} ( \Pi_L,\pi_{\zeta,\chi,c,\omega} ) 
 \simeq 
 \begin{cases}
 \rho_{\zeta,\chi,c,\omega} \otimes \tau_{\zeta,\chi,c,\omega} \quad  & 
 \textrm{if $\zeta_L \zeta^{-1} \in \mu_{(q-1)/n_q} (K)$}, \\
 0 \quad & \textrm{otherwise}
 \end{cases}
\]
as $D^{\times} \times W_K$-representations.
\end{lem}
\begin{proof}
For $g \in H_L \backslash G/G_1$, we choose an element $\tilde{g} \in G_2$
whose image in $\ol{H_L} \backslash G_2$ corresponds to $g$
under the natural isomorphism 
$H_L \backslash G/G_1 \simeq \ol{H_L} \backslash G_2$. 
We put $H^{\tilde{g}}=\tilde{g}^{-1}H\tilde{g}$. 
Let $\Pi_{\fX^L}^{\tilde{g}}$ denote the representation of 
$H^{\tilde{g}}$ which is the conjugate of $\Pi_{\fX^L}$ 
by $\tilde{g}$.
Then, we have 
\begin{equation}\label{g1}
 \Pi_L|_{G_1} \simeq \bigoplus_{g \in H_L \backslash G/G_1} 
 \mathrm{c\mathchar`-Ind}_{H^{\tilde{g}}}^{G_1} 
 \Pi_{\fX^L}^{\tilde{g}}
 \simeq \bigoplus_{\ol{H_L} \backslash G_2} 
 \mathrm{c\mathchar`-Ind}_{H}^{G_1}\Pi_{\fX^L}
\end{equation}
as $G_1$-representations by the Mackey decomposition, 
since $H^{\tilde{g}}=H$ and 
$\Pi_{\fX^L} \simeq \Pi_{\fX^L}^{\tilde{g}}$ 
as $H$-representations. 
By \eqref{g1}, Lemma \ref{pro} and the Frobenius reciprocity, 
we have the first claim and 
\begin{equation}\label{g2}
 \mathrm{Hom}_{G_1} 
 ( \Pi_L,\pi_{\zeta,\chi,c,\omega} ) \simeq 
 \bigoplus_{\ol{H_L} \backslash G_2} 
 \mathrm{Hom}_{H } (\Pi_{\fX^L}, 
 \pi_{\zeta,\chi,c,\omega} ). 
\end{equation}
On the other hand, the natural morphism 
$\Pi_{\fX^L} \to \Pi_L |_{H_L}$ induces the morphism 
\begin{equation}\label{Pipr}
 \mathrm{Hom}_{G_1} 
 ( \Pi_L,\pi_{\zeta,\chi,c,\omega} ) \to 
 \mathrm{Hom}_{H } (\Pi_{\fX^L}, 
 \pi_{\zeta,\chi,c,\omega} ) 
\end{equation}
as $\ol{H_L}$-representations. 
By the construction, 
\eqref{Pipr} coincides with the projection to 
the component labeled by 
$1 \in \ol{H_L} \backslash G_2$ in \eqref{g2}. 

If $\zeta_L \zeta^{-1} \notin \mu_{(q-1)/n_q} (K)$, 
the claim follows from \eqref{g2} and Lemma \ref{pro}.1. 
Now, assume that $\zeta_L \zeta^{-1} \in \mu_{(q-1)/n_q} (K)$. 
We may assume that $\zeta_L =\zeta$ by Lemma \ref{ttbij} 
and Remark \ref{onlyL}. 
By Lemma \ref{pro} and the Frobenius reciprocity,
we obtain a non-zero map
\begin{equation}\label{vap}
 \mathrm{Ind}_{L^{\times} U_D^1 \times W_{L}}^{\ol{H_L}} 
 \bigl( 
 \theta_{\zeta,\chi,c,\omega} \otimes \mu_{\zeta}^{-1}\xi_{\zeta,\chi,c,\omega} 
 \bigr) 
 \to \mathrm{Hom}_{H }(\Pi_{\fX^L},\pi_{\zeta,\chi,c,\omega}). 
\end{equation}
By applying 
$\mathrm{Ind}_{\ol{H_L}}^{G_2}$ to the map \eqref{vap}, 
we obtain a non-zero map 
\begin{equation}\label{ff}
 \rho_{\zeta,\chi,c,\omega} \otimes \tau_{\zeta,\chi,c,\omega}
 \to \mathrm{Ind}_{\ol{H_L}}^{G_2}\mathrm{Hom}_{H}
 (\Pi_{\fX^L},\pi_{\zeta,\chi,c,\omega}).
\end{equation}
We have $\dim \rho_{\zeta,\chi,c,\omega}=(q^n-1)/(q-1)$ and 
$\dim\tau_{\zeta,\chi,c,\omega}=n$. 
Hence, 
both sides of \eqref{ff} are $n(q^n-1)/(q-1)$-dimensional 
by Lemma \ref{pro}.1 and Lemma \ref{Hindex}. 
Since $\rho_{\zeta,\chi,c,\omega} \otimes \tau_{\zeta,\chi,c,\omega}$ 
is an irreducible representation of $G_2$, 
we see that \eqref{ff} is an isomorphism 
as $G_2$-representations.
On the other hand, by \eqref{Pipr} and the Frobenius reciprocity, 
we have a non-zero map
\begin{equation}\label{tti}
\Hom_{G_1}  (\Pi_L,\pi_{\zeta,\chi,c,\omega} ) \to 
\Ind^{G_2}_{\ol{H_L}} \mathrm{Hom}_{H}
(\Pi_{\fX^L},\pi_{\zeta,\chi,c,\omega} ). 
\end{equation}
We see that \eqref{tti} is an isomorphism, 
since the right hand side 
is an irreducible representation of $G_2$ and
both sides have the same dimension by \eqref{g2}. 
Hence, the claim follows from 
the isomorphisms \eqref{ff} and \eqref{tti}. 
\end{proof}
\begin{thm}\label{Pireal}
Let $\mathrm{LJ}$ be the inverse of $\mathrm{JL}$ in Proposition 
\ref{expJL}.
We set 
$\Pi=\bigoplus_{L \in T(K,n)}\Pi_L$. 
Let $\pi$ be a smooth irreducible supercuspidal representation of 
$\iGL_n (K)$. 
Then, we have 
\[
 \mathrm{Hom}_{\mathit{GL}_n(K)} ( \Pi,\pi ) \simeq 
 \begin{cases}
  \mathrm{LJ}(\pi) \otimes \mathrm{LL}(\pi) & 
  \textrm{if $\pi$ is essentially simple supercuspidal,} \\ 
  0 & \textrm{otherwise} 
 \end{cases}
\]
as $D^{\times} \times W_K$-representations.
\end{thm}
\begin{proof}
This follows from Proposition \ref{expJL} 
and Lemma \ref{map}, because 
every essentially simple supercuspidal representation is 
isomorphic to 
$\pi_{\zeta,\chi,c,\omega}$ 
for some 
$\zeta \in \mu_{q-1} (K)$, $\chi \in (k^{\times})^{\vee}$, 
$c \in \overline{\mathbb{Q}}_{\ell}^{\times}$ and 
a smooth character 
$\omega \colon K^{\times} \to \ol{\bQ}_{\ell}^{\times}$.
\end{proof}

\subsection{Proof of Proposition \ref{ky}}\label{PrProp}
Assume that $p$ is odd. 
Let $E$ be a finite separable extension of $K$. 
Let $\mathfrak{p}_E$ be the maximal ideal of $\mathcal{O}_E$. 
We write $q_E$ 
for the cardinality of $k_E$. 
Let 
$\psi_0^E \colon k_E \to \overline{\mathbb{Q}}_{\ell}^{\times}$ 
be a non-trivial character. 
We put 
\[
 \mathfrak{g}(\psi_0^E)=
 \sum_{x \in k_E^{\times}}\left(\frac{x}{k_E}\right)\psi_0^E (x) \quad 
 \textrm{and} \quad 
 \mathfrak{m}(\psi_0^E)=\mathfrak{g}(\psi_0^E)/q_E^{1/2}. 
\]
Let
$\psi^E \colon E \to \overline{\mathbb{Q}}_{\ell}^{\times}$ 
be a character such that 
$\psi^E (x)= \psi_0^E (\bar{x})$ for $x \in \cO_E$. 
We write 
$\mathfrak{m}(\psi^E)$ for $\mathfrak{m}(\psi_0^E)$. 
Then, we have 
\begin{equation}\label{gr}
 \mathfrak{m}(\psi^E)^2=
 \left(\frac{-1}{k_E}\right)
\end{equation}
by \cite[Sommes trig. 4.4]{DelCoet}.
In particular, 
$\mathfrak{m}(\psi^E)$ is a fourth root of unity. 

By \cite[Proposition 4.5]{BHestII}, we have 
\begin{equation}\label{qua}
 \mathfrak{g}(\nu_{n-1},\psi) = 
 \left(\frac{\det \nu_{n-1}}{q}\right)\mathfrak{g}(\psi)^{n-1}.
\end{equation}
Here, we restate Proposition \ref{ky}. 
\begin{prop}
Assume that $p$ is odd. 
Let $L$ be a totally tamely ramified extension of $K$ 
of degree $n$. 
Then, we have 
\begin{equation}\label{uy}
 \lambda_{L/K}(\psi_K)
 =
 \mathfrak{g}(\nu_{n-1},\psi) q^{-\frac{n-1}{2}}.
\end{equation}
\end{prop}
\begin{proof}
By \eqref{qua}, it suffices to show 
\begin{equation}\label{uy2}
\lambda_{L/K}(\psi_K)=
\left( \frac{\det \nu_{n-1}}{q}\right)
\mathfrak{m}(\psi_K)^{n-1}.
\end{equation}
Using Lemma \ref{detnu} and \eqref{gr}, 
the equality \eqref{uy2}
is rewritten as follows:
\begin{equation}\label{son}
\lambda_{L/K}(\psi_K) =
\begin{cases}
 \bigl( \frac{n}{q} \bigr) 
 \mathfrak{m}(\psi_K)^{n-1}\quad  & \textrm{if $n$ is odd}, \\[0.2cm]
 \bigl( \frac{n/2}{q} \bigr) 
 \mathfrak{m}(\psi_K)^{n-1} \quad & \textrm{if $n$ is even}. 
\end{cases}
\end{equation}

First, we consider the case where $n$ is odd.
In this case, 
$\lambda_{L/K}(\psi_K)$ equals the Jacobi symbol 
$\left(\frac{q}{n}\right)$ by \cite[Theorem 2.1(1)]{BHestII}. 
Since $n-1$ is even, we have 
\[
 \mathfrak{m}(\psi_K)^{n-1}
 =\left(\frac{-1}{q}\right)^{\frac{n-1}{2}} 
 =(-1)^{\frac{n-1}{2}\frac{q-1}{2}} 
\] 
by \eqref{gr}. 
Hence, \eqref{son} for the odd case 
is equivalent to
\[
 \biggl(\frac{q}{n}\biggr)=
 \biggl(\frac{n}{q}\biggr)(-1)^{\frac{n-1}{2}\frac{q-1}{2}}.
\] 
This follows from the quadratic reciprocity law. 
Hence, if $n$ is odd, 
we have proved the equality \eqref{uy2}. 

We consider the case where $n$ is even.
Let $K'$ be the unique quadratic subextension 
of $L$ over $K$. 
Then, we have 
\[
 \lambda_{L/K}(\psi_K)=\lambda_{L/K'}(\psi_{K'})
 \lambda_{K'/K}(\psi_K)^{\frac{n}{2}}
\]
by \cite[(1.5.2)]{BHestII}. 
Furthermore, we have 
$\lambda_{K'/K}(\psi_K)=\mathfrak{m}(\psi_K)$ 
by \cite[Lemma 1.5(3)]{BHestII}.
Hence, we acquire 
\begin{equation}\label{ind}
 \lambda_{L/K}(\psi_K)=
 \lambda_{L/K'}(\psi_{K'})
 \mathfrak{m}(\psi_K)^{\frac{n}{2}}.
\end{equation}
On the other hand, we have 
\begin{equation}\label{m}
 \mathfrak{m}(\psi_{K'})=q^{-\frac{1}{2}}
 \sum_{x \in k^{\times}} \biggl( \frac{x}{k} \biggr) 
 \psi_K^0(2x)
 =\biggl( \frac{2}{q} \biggr) 
 \mathfrak{m}(\psi_K).
\end{equation}

Let $v_2(n)$ denote the $2$-adic valuation of $n$. 
We prove \eqref{uy2} by induction on $v_2(n)$.
First, 
we prove \eqref{uy2} when $v_2(n)=1$.
In this case, by \eqref{son} for the odd case, 
\eqref{ind} and \eqref{m}, we have
\[
 \lambda_{L/K}(\psi_K)=
 \left(\frac{n/2}{q}\right)\mathfrak{m}(\psi_{K'})^{\frac{n}{2} -1}
 \mathfrak{m}(\psi_K)^{\frac{n}{2}}=
 \left(\frac{n/2}{q}\right)\mathfrak{m}(\psi_K)^{n-1}.
\]
Hence, we obtain \eqref{uy2} in this case.

Now, we consider the case where $v_2(n)>1$. 
By the induction hypothesis 
and \eqref{m}, 
we have 
\[
 \lambda_{L/K'}(\psi_{K'})=
 \left(\frac{n/4}{q}\right)\mathfrak{m}(\psi_{K'})^{\frac{n}{2} -1}=
 \left(\frac{n/2}{q}\right)\mathfrak{m}(\psi_K)^{\frac{n}{2} -1}. 
\]
Hence, the claim follows from \eqref{ind}.
\end{proof}


\noindent
Naoki Imai\\ 
Graduate School of Mathematical Sciences, 
The University of Tokyo, 3-8-1 Komaba, Meguro-ku, 
Tokyo, 153-8914, Japan\\ 
naoki@ms.u-tokyo.ac.jp \vspace*{10pt}

\noindent
Takahiro Tsushima\\ 
Department of Mathematics and Informatics, 
Faculty of Science, Chiba University, 
1-33 Yayoi-cho, Inage, Chiba, 263-8522, Japan\\
tsushima@math.s.chiba-u.ac.jp

\end{document}